\documentclass[12pt]{article}

\usepackage{mathrsfs}
\usepackage{amsmath,extpfeil}
\usepackage{stmaryrd}
\usepackage{mathrsfs}
\usepackage{url}
\usepackage{indentfirst}
\usepackage{enumerate}
\usepackage{amsmath,amsfonts,amssymb,amsthm}
\usepackage{amsmath,amssymb,amsthm,amscd}
\usepackage{graphicx,mathrsfs}
\usepackage{appendix}
\usepackage[numbers,sort&compress]{natbib}
\usepackage{color}
\usepackage[colorlinks, linkcolor=blue, citecolor=blue]{hyperref}

\usepackage{sidecap}
\usepackage{float}
\usepackage{extarrows}
\usepackage{booktabs}
\usepackage{verbatim}
\usepackage[usenames,dvipsnames]{xcolor}

\newtheorem{theorem}{Theorem}[section]
\newtheorem{lemma}[theorem]{Lemma}
\newtheorem{corollary}[theorem]{Corollary}
\newtheorem{proposition}[theorem]{Proposition}

\theoremstyle{definition}

\newtheorem{defn}{Definition}

\def\XXint#1#2#3{{\setbox0=\hbox{$#1{#2#3}{\int}$}
         \vcenter{\hbox{$#2#3$}}\kern-.5\wd0}}

\def\R{\mathbb{R}}

\numberwithin{equation}{section}

\allowdisplaybreaks[4]

\begin{document}

\title{Weak stability of the sum of two-solitary waves for Half-wave equation}

\author{Yuan Li }
%\thank
\date{}
\maketitle

\begin{abstract}
In this paper, we consider the subcritical half-wave equation in one dimension. Let $R_k(t,x)$, $k=1,2$, represent two-solitary wave solutions of the half-wave equation, each with different translations $x_1,x_2$. We prove that if the relative distance $x_2-x_1$ between the two solitary waves is large enough, then the sum of $R_k(t)$ is weakly stable. Our proof relies on an energy method and the local mass monotonicity property. Unlike the single-solitary wave or NLS cases, the interactions between different waves are significantly stronger here. To establish the local mass monotonicity property, as well as to analyze non-local effects on localization functions and non-local operator $D$, we utilize the Carlder\'on estimate and the integral representation formula of the half-wave operator.


\end{abstract}

\section{Introduction and Main Result}
\noindent
We consider the half-wave equation of the form
\begin{equation}\label{equ-hf}
\begin{cases}
i\partial_tu-Du+|u|^{p-1}u=0,~~~(t,x)\in \mathbb{R}\times\mathbb{R}^N,\\
u(0)=u_0\in H^s,
\end{cases}
\end{equation}
where $D$ is defined via the Fourier transform by  $\widehat{(Df)}(\xi)=|\xi|\hat{f}(\xi)$, denotes the first-order nonlocal fractional derivative, $s\in(\frac{1}{2},1)$ and $p\in(1,\infty)$. Evolution problems like \eqref{equ-hf} arise in a variety of physical contexts, such as turbulence phenomena \cite{MMT1997JNS,CMMT2001PD}, wave propagation \cite{W1987CPDE}, continuum limits of lattice system \cite{KLS2013CMP} and models for gravitational collapse in astrophysics \cite{IP2014JFA,FL2013Acta,ES1983MMAS}. For further background on half-wave equation or fractional Schr\"odinger model in mathematics, numerics, and physics, one can see \cite{ES2007CPAM,FJL2007CMP,KSM2014PRSL,CHHO2013FE} and the references therein.

Let us review some basic properties of the equation \eqref{equ-hf}. The Cauchy problem \eqref{equ-hf} is an infinite-dimensional Hamiltonian system, which has the following three conservation quantities:\\
Energy:
\begin{align}\label{energy}
    E(u(t))=\frac{1}{2}\int_{\mathbb{R}^N}\left|D^{\frac{1}{2}}u(t,x)\right|^2dx-\frac{1}{p+1}\int_{\mathbb{R}^N}|u(t,x)|^{p+1}dx=E(u_0);
\end{align}
Mass:
\begin{align}\label{Mass}
    M(u(t))=\int_{\mathbb{R}^N}|u(t,x)|^2dx=M(u_0);
\end{align}
Momentum:
\begin{align*}
    P(u(t))=(i\nabla u(t,x),u(t,x))_{L^2}=P(u_0).
\end{align*}
Here we regard $L^2$ as a real Hilbert space equipped with the inner product
\begin{align*}
    (u,v)_{L^2}=\Re\int u\bar{v}dx.
\end{align*}
Equation \eqref{equ-hf} also admits the following symmetries.
\begin{itemize}
    \item Phase invariance: If $u(t,x)$ satisfies \eqref{equ-hf}, then for any $\gamma_0\in\mathbb{R}$, $u(t,x)e^{i\gamma_0}$  also satisfies \eqref{equ-hf};

    \item  Translation invariance: If $u(t,x)$ satisfies \eqref{equ-hf}, then for any $x_0\in\mathbb{R^N}$ and $t_0\in\mathbb{R}$, $u(t-t_0,x-x_0)$  also satisfies \eqref{equ-hf};

    \item  Scaling invariance: If $u(t,x)$ satisfies \eqref{equ-hf}, then for any $\lambda_0\in\mathbb{R}$, $\lambda_0^{\frac{1}{p-1}}u(\lambda_0t,\lambda_0x)$  also satisfies \eqref{equ-hf}.
\end{itemize}
It leaves invariant the norm in the homogeneous Sobolev space $\dot{H}^{s_c}$, where $s_c=\frac{N}{2}-\frac{1}{p-1}$. If $s_c<0$, the problem is mass-subcritical; if $s_c=0$, it is mass-critical and if $0<s_c<1$, it is mass super-critical and energy subcritical or just intercritical.

It is well known that the Cauchy problem is locally well-posed in Sobolev space $H^s$, for $s>\frac{1}{2}$. Specifically, Krieger, Lenzmann and Rapha\"el \cite{KLR2013ARMA} showed that equation \eqref{equ-hf} is locally well-posed in $H^s(\mathbb{R})$, $s>\frac{1}{2}$, and established local existence in $H^{\frac{1}{2}}(\mathbb{R})$ with $p=3$. For the higher dimensional case, Bellazzini, Georgiev and Visciglia \cite{BGV2018MN} proved that equation \eqref{equ-hf} is locally well-posed in $H^1_{rad}(\mathbb{R}^N)$, with $s_c<1$ and $N\geq2$; Furthermore, Hidano and Wang \cite{HW2019SM} improved this result and established local existence in the space $H^s_{rad}(\mathbb{R}^N)$,
$s\in(\frac{1}{2},1)$ and $N\geq2$, as well as in $H^s(\mathbb{R}^N)$, $s\in\left(\max\left\{\frac{N-1}{2},\frac{N+1}{4}\right\},p\right)$ and $s\geq s_c$, where $p>\max\left\{s_c,\frac{N-1}{2},\frac{N+1}{4},1\right\}$.
% {\color{red}
% In particular, for any $u_0\in H^s(\mathbb{R}^N)$, $s>\frac{1}{2}$, there exists $T>0$ and a unique maximal solution $u\in C([0,T), H^s)$ of \eqref{equ-hf} on $[0,T)$. Moreover, either $T=+\infty$ or $T<\infty$ and then $\lim_{t\to T}\|D^{\frac{1}{2}}u(t)\|_{L^2}\to+\infty$.}
For more details on the half-wave equation \eqref{equ-hf}, one can see \cite{D2018DCDS,FGO2019JMPA} for local/global well-posedness, \cite{FGO2018DPDE,I2016PAMS,GL2021JFA,GL2022CPDE,GL2025JDE} for finite-time blow up, \cite{BGV2018MN} for stability/instability of ground states, \cite{FS2024JDE,BGLV2019CMP,GLPR2018APDE} for non-scattering traveling waves with arbitrary small mass, \cite{D2018DCDS} for ill-posedness for low-regularity data, and \cite{FGO2019JMPA} for the proof of various prior estimates.

This paper is concerned with questions related to special solutions of equation \eqref{equ-hf}, called the solitary wave solution, which are fundamental in the dynamics of the equation. For $\omega>0$, let
\begin{align}\label{standing}
    u(t,x)=e^{i\omega t}Q_{\omega}(x)
\end{align}
be the $H^{\frac{1}{2}}(\mathbb{R}^N)$ solution of \eqref{equ-hf} if $Q_{\omega}:\mathbb{R}^N\to\mathbb{R}$ is an $H^{\frac{1}{2}}(\mathbb{R}^N)$ solution of
\begin{align}\label{equ-elliptic}
    D Q_{\omega}+\omega Q_{\omega}=Q_{\omega}^p.
\end{align}

The first question concerning the solitary wave solutions of \eqref{equ-hf} is whether they are stable by perturbation of the initial data in the energy space, that is, whether or not the following property is satisfied.
\begin{defn}
A solitary wave solution of the form \eqref{standing} is weakly orbitally stable if for all $\epsilon>0$, there exists $\delta>0$ such that if
\begin{align*}
\|u_0-Q_{\omega_0}(\cdot-x_0)e^{i\gamma_0}\|_{H^{\frac{1}{2}}}\leq\delta,~~u_0\in H^s(\mathbb{R}^N),~~s\in\left(\frac{1}{2},1\right),
\end{align*}
then for all $t\in\mathbb{R}$, there exist $x(t)\in\mathbb{R}^N$ and $\gamma(t)\in\mathbb{R}$ such that  $u(t)$    satisfies
\begin{align*}
\|u(t)-Q_{\omega_0}(\cdot-x(t))e^{i\gamma(t)}\|_{H^{\frac{1}{2}}}\leq\epsilon.
\end{align*}
where $u(t)$ is the unique global solution of \eqref{equ-hf} associated with the initial data $u_0$.
\end{defn}
Due to the invariances of the half-wave equation, whether or not this property is satisfied does not depend on $x_0$ and $\gamma_0$. To the best of our knowledge, Bellazzini, Georgiev and Visciglia \cite{BGV2018MN} proved the stability or instability of the ground state to the half-wave equation \eqref{equ-hf}  by the classical argument of Cazenave-Lions (see \cite{CL1982CMP}). On the other hand, Cao, Su, and Zhang \cite{CSZ2024JLMS} constructed the multi-bubble blow-up solutions to the mass critical half-wave equation in one dimension; Meanwhile, G\'erard, Lenzmann, Pocovnicu and Rapha\"el \cite{GLPR2018APDE} constructed asymptotic global-in-time compact two-soliton solutions of \eqref{equ-hf} that have an arbitrarily small $L^2$-norm, where $p=3$ and $N=1$.

% {\color{red} \rule[-10pt]{16cm}{0.3em}}

In this paper, our aim is to study the stability of the sum of two-solitary waves of the equation \eqref{equ-hf} by using the expansion of the conservation laws around the solitary wave. Now we state our main result.
\begin{theorem}\label{Thm1}(Stability of the sum of two solitary waves in one dimension).
Let $N=1$, $1<p<3$, and $\frac{1}{2}<s<1$. Assume that  there exist $Q_{\omega_k^0}\in H^\frac{1}{2}(\mathbb{R})$, where $\omega_k^0>0$ and $k=1,2$,  positive solutions of \eqref{equ-elliptic} satisfying
 \begin{align}\label{condition}
    \frac{d}{d\omega}\int_{\mathbb{R}}Q_{\omega}^2(x)dx\Big|_{\omega=\omega_k^0}>0.
\end{align}
For  $k\in\{1,2\}$, let $x_k^0\in\mathbb{R},\gamma_k^0\in\mathbb{R}$.  There exist  $\sigma_0>1$, $A_0>0$ and $\alpha_0>0$ such that for any $u_0\in H^s(\mathbb{R}^N)$, $\sigma>\sigma_0$, and $0<\alpha<\alpha_0$ if
\begin{align*}
    \left\|u_0-\sum_{k=1}^2Q_{\omega_k^0}(\cdot-x_k^0)e^{i\gamma_k^0}\right\|_{H^\frac{1}{2}}\leq\alpha,
\end{align*}
and
\begin{align}\label{def:sig0}
|x_1^0-x_2^0|\geq \sigma.
\end{align}
Then the solution $u(t)$ of \eqref{equ-hf} is globally defined in $H^s$, $s\in(\frac{1}{2},1)$ for all $t\geq0$, and there exist $C^1$-functions $\gamma_j(t)\in\mathbb{R}$, $x_j(t)\in\mathbb{R}$, $j=1,2$,  such that for all $t\geq0$,
\begin{align}\label{a:2}
     \left\|u(t)-\sum_{k=1}^2Q_{\omega_k^0}(\cdot-x_k(t))e^{i\gamma_k(t)}\right\|_{H^\frac{1}{2}}\leq A_0\left(\alpha+\frac{1}{\langle \sigma\rangle}\right),
\end{align}
where $A_0>0$ is a constant.
\end{theorem}

{\bf Comments:}

1. The equation \eqref{equ-hf} ($N=1$ and $1<p<3$) is well-posed in $H^s$, $s>\frac{1}{2}$, can be found in Appendix, see Lemma \ref{Lemma:local}. In particular, if $N=1$ and $p=3$, the well-posedness result in $H^{\frac{1}{2}}$ can be deduced in a verbatim fashion as for the
so-called cubic Szeg\"o equation,  see \cite[Appendix D]{KLR2013ARMA}. However, when $p\neq3$, the well-posedness result in $H^{\frac{1}{2}}(\mathbb{R})$ is still unknown.

2. Assumption \eqref{condition}. Notice that by equation \eqref{equ-elliptic}, the function $S_{\omega}\in H^{\frac{1}{2}}$ defined by $S_\omega=\frac{\partial}{\partial\omega}Q_\omega$ satisfies $L_{\omega}^+S_\omega=-Q_\omega$, where
\[L_{\omega}^+=D+\omega-pQ_\omega^{p-1}.\]
Therefore, condition \eqref{condition} is equivalent to
\[(S_{\omega_k^0},Q_{\omega_k^0})>0.\]
Moreover, it turns out that \eqref{condition} implies
\begin{align*}
    \inf\left\{\frac{(L^+_{\omega_0}v,v)}{(v,v)},~v\in H^{\frac{1}{2}},~(v,Q_{\omega_0})=(v,Q^{\prime}_{\omega_0})=0\right\}>0.
\end{align*}
This implies that $L_{\omega}^+$ is coercivity in $H^{\frac{1}{2}}$ (see Lemma \ref{lemma:linear:operator}).

% (3) Assumption \eqref{assume:2}. The condition \eqref{assume:2} is a technical assumption.  This means that a certain relationship must exist between the maximum and minimum distances of any two waves. In addition, the behavior of this distance is as $t^{\frac{2}{1-2\delta}}$ when $t$ is sufficiently large.

3. The structure of the problem is similar to the following two types of equations:

3.1). Mass subcritical nonlinear Schr\"odinger (NLS) equation
\begin{eqnarray*}
iu_t+\Delta u+|u|^{p-1}u=0.
\end{eqnarray*}
Martel, Merle and Tsai \cite{MMT2006Duke} proved the stability of the sum of two-solitary waves of this equation.

3.2). For the equation
\begin{equation*}
    u_t+\partial_x|D|^\alpha u+(u^p)_x=0,~~\text{for}~~p=2,3,4.
\end{equation*}
For $\alpha=2$ in the former equation, which corresponds to the subcritical generalized Korteweg-de Vires (gKdV) equation, Martel, Merle and Tsai \cite{MMT2002CMP} obtained the stability of the sum of $K$ solitons of gKdV equations; For $\alpha=1$, which is the generalized Benjamin-Ono (BO) equation, Kenig and Martel \cite{KM2009MRI} studied the asymptotic stability of the solitons,  and  Gustafson,  Takaoka and Tsai \cite{GTT2009JMP} studied the stability of the sum of  $K$-soliton solution for BO equation; For $\alpha\in(1,2)$, which is the fractional modified Korteweg-de Vries equation, Eychenne and Valet \cite{EV2023JFA} proved that the existence of solution behaving in large time as a sum of two strongly interacting solitary waves with different signs.

However, in the scenario of multi-solitary waves for the half-wave equation, different waves exhibit the strongest interaction. Therefore, we introduce a property of local monotonicity, related to mass, which is similar to the $L^2$- monotonicity property for the KdV equation used in \cite{MMT2002CMP} (see Lemma \ref{lemma:mass:loc}). Additionally, to deal with the non-local term, we utilize the Carder\'on estimate and the integration representation formula of the half-wave operator (see Lemma \ref{lemma:mass:loc} and \ref{lemma:coer:m2}).

This paper is organized as follows. In Section 2, we give some fundamental properties and Lemmas. In Section 3, we study the weak stability of a single solitary wave. In Section 4, we prove Theorem \ref{Thm1} and obtain the weak stability of the sum of two-solitary waves.

\section{Preliminaries}
In this section, we aim to introduce some fundamental properties and useful lemmas.
\begin{lemma}[\cite{DPV2012BSM}]
Let $s\in(0,1)$.
Then, for any $f\in \mathscr{S}(\mathbb{R})$,
\begin{align*}
    (-\Delta)^sf(x)&=C(s) P.V.\int_{\mathbb{R}}\frac{f(x+y)-f(x)}{|y|^{1+2s}}dy\notag\\
&=-\frac{1}{2}C(s)\int_{\mathbb{R}}\frac{f(x+y)+f(x-y)-2f(x)}{|y|^{1+2s}}dy,
\end{align*}
where the normalization constant is given by
\begin{align}\label{constant}
    C(s)=\left(\int_{\mathbb{R}}\frac{1-cosx}{|x|^{1+2s}}dx\right)^{-1}.
\end{align}
\end{lemma}
The following lemma provides useful formulas for the fractional
Laplacian of the product of two functions.
\begin{lemma}[\cite{BWZ2017ANS}]\label{lem-formula}
Let $s\in(0,1)$ and $f, g\in \mathscr{S}(\mathbb{R})$.
Then, we have
\begin{align}
    \label{id-fra2}
(-\Delta)^{s}(fg)-(-\Delta)^sfg
=C(s)\int\frac{f(x+y)(g(x+y)-g(x))-f(x-y)(g(x)-g(x-y))}{|y|^{1+2s}}dy,
\end{align}

and
\begin{align*}
(-\Delta)^{s}(fg)-(-\Delta)^sfg-f(-\Delta)^sg
=-C(s)\int\frac{(f(x+y)-f(x))(g(x+y)-g(x))}{|y|^{1+2s}}dy,
\end{align*}
where $C(s)$ is defined as in \eqref{constant}.
\end{lemma}

We recall the following commutator estimate.
\begin{lemma}(\cite[Theorem 2]{C1965PNA})\label{lemma:commu:1}
For $\nabla f\in L^\infty(\mathbb{R})$ and $g\in \mathscr{S}(\mathbb{R})$,
we have
\begin{align*}
 \|D(fg)-fDg\|_{L^2}\leq C\|\nabla f\|_{L^\infty}\|g\|_{L^2}.
\end{align*}

\end{lemma}

Next we give the following existence and uniqueness results of \eqref{equ-elliptic}.
\begin{lemma} (see \cite{FL2013Acta,FLS2016CPAM})
If $\omega>0$ and $s_c<1$. Then the following hold.

(i) Problem \eqref{equ-elliptic} has a positive, radial symmetry ground state solution $Q_{\omega}$;

(ii) The ground state solution $Q_{\omega}$ is unique up to translation;

(iii) There exists $C>0$ such that
\begin{align}\label{decay}
    |Q_{\omega}(x)|\leq\frac{C}{\langle x\rangle^{N+1}}~~\text{for all}~~x\in\mathbb{R}^N.
\end{align}

(iv) Now define the operator
\begin{align*}
    L_{\omega}=(L_{\omega}^+,L_{\omega}^-),
\end{align*}
where
\begin{align*}
    L_{\omega}^+v=D v+{\omega} v-pQ_{\omega}^{p-1}v,~~L_{\omega}^-v=D v+{\omega} v-Q_{\omega}^{p-1}v,~~v\in H^{\frac{1}{2}}(\mathbb{R}^N).
\end{align*}
Then we have
\begin{align*}
    \ker L_{\omega}^+=span\{\nabla Q\},~~ \ker L_{\omega}^-=span\{ Q\}.
\end{align*}
\end{lemma}

The following lemma is the positive property of the operator $L_{\omega}$.
\begin{lemma}\label{lemma:linear:operator}
Assuming that $\omega>0$, the following hold.

(i) There exists $\lambda_+>0$ such that for any real-valued function $v\in H^{\frac{1}{2}}(\mathbb{R}^N)$,
\begin{align*}
    (L_{\omega}^+v,v)\geq\lambda_+\|v\|^2_{H^{\frac{1}{2}}(\mathbb{R}^N)}~~\text{for}~~(v,Q_{\omega})=(v,\nabla Q_{\omega})=0.
\end{align*}

(ii) There exists $\lambda_->0$ such that for any real-valued function $v\in H^{\frac{1}{2}}(\mathbb{R}^N)$,
\begin{align*}
    (L_{\omega}^-v,v)\geq\lambda_-\|v\|^2_{H^{\frac{1}{2}}(\mathbb{R}^N)}~~\text{for}~~(v,Q_{\omega})=0.
\end{align*}
\end{lemma}
\begin{proof}
The proof of this result was given by \cite{KLR2013ARMA} for mass critical in one dimensional case. For higher dimensional, one can see our previous results \cite{GL2021JFA,GL2022CPDE}. For general case, exactly the same arguments apply to prove this Lemma.
\end{proof}

We also recall some variational properties of $Q_{\omega_0}$. Define
\begin{align}\label{def:functional}
    G_{\omega}(u)=E(u)+\frac{\omega}{2}\int|u|^2.
\end{align}
\begin{lemma}\label{lemma:energy:realtion}
For $\omega,\omega_0>0$. Let $Q_{\omega_0}$ be the ground state of equation \eqref{equ-elliptic} with $\omega=\omega_0$ and $Q_\omega$ be defined as \eqref{equ-elliptic}. For $\eta=\eta_1+i\eta_2\in H^{\frac{1}{2}}(\mathbb{R})$ is small, we have
    \begin{align*}
        G_{\omega_0}(Q_{\omega_0}+\eta)=G_{\omega_0}(Q_{\omega_0})+(L_{\omega_0}^+\eta_1,\eta_1)+(L_{\omega_0}^-\eta_2,\eta_2)+\|\eta\|_{H^\frac{1}{2}}^2\beta\left(\|\eta\|_{H^\frac{1}{2}}\right),
    \end{align*}
with $\beta(\epsilon)\to0$ as $\epsilon\to0$. In particular, for $\omega$ close to $\omega_0$,
\begin{align*}
    G_{\omega_0}(Q_{\omega})=G_{\omega_0}(Q_{\omega_0})+(\omega-\omega_0)^2(L_{\omega_0}^+S_{\omega_0},S_{\omega_0})+|\omega-\omega_0|^2\beta(|\omega-\omega_0|),
\end{align*}
where $S_\omega=\frac{\partial}{\partial\omega}Q_\omega$ and $L_{\omega}^+S_\omega=-Q_\omega$.
Moreover,
\begin{align*}
    |G_{\omega_0}(Q_{\omega_0})-G_{\omega_0}(Q_{\omega})|\leq C|\omega-\omega_0|^2.
\end{align*}
\end{lemma}
% \begin{remark}
% The assumption \eqref{condition} implies that
% \begin{align*}
%     \frac{1}{2}\frac{d}{d\omega}\int Q_{\omega}^2dx\big|_{\omega=\omega_0}=(S_{\omega_0},-Q_{\omega_0})=-(L_{\omega_0}^+S_{\omega_0},S_{\omega_0})>0.
% \end{align*}
% \end{remark}
\begin{proof}[\bf Proof of Lemma \ref{lemma:energy:realtion}.]
By a similar argument as the classical nonlinear Schr\"odinger equation (see \cite[Section 2]{W1986CPAM} or \cite[Lemma 2.3]{MMT2006Duke}), one can obtain this result. Here we omit it.

\end{proof}

\subsection{Decomposition of a solution}
Let $\omega_1^0,\omega_2^0\in\mathbb{R}$. For $\alpha>0$, $\gamma_0\in\mathbb{R}$ and $\sigma_0$ are given as in Theorem \ref{Thm1}, we consider the $H^\frac{1}{2}$-neighborhood of size $\alpha$ of the sum of two solitary waves with parameter $\{\omega_k^0\}$, located at distances satisfies the assumption \eqref{def:sig0} and
\begin{align*}
    B(\alpha,\sigma_0)=\left\{u\in H^s,~\inf\left\|u(\cdot)-\sum_{k=1}^2Q_{\omega_k^0}(\cdot-y_k)e^{i\gamma_0}\right\|_{H^\frac{1}{2}}<\alpha\right\}.
\end{align*}

Now we give the following Lemma.
\begin{lemma}\label{lemma:decom:unique}
There exist $\alpha_1, \sigma_1, C_1>0$, and for any $k\in\{1,2\}$, there exist unique $C^1$-functions $(\omega_k,x_k,\gamma_k):B(\alpha_1,\sigma_1)\to(0,\infty)\times\mathbb{R}\times\mathbb{R}$ such that if $u\in B(\alpha_1,\sigma_1)$ and if one defines
\begin{align*}
    \epsilon(x)=u(x)-\sum_{k=1}^2Q_{\omega_k}(\cdot-x_k)e^{i\gamma_k},
\end{align*}
then for all $k=1,2$,
\begin{align*}
    \Re\int Q_{\omega_k}(\cdot-x_k)e^{i\gamma_k}\bar{\epsilon}(x)dx=\Im\int Q_{\omega_k}(\cdot-x_k)e^{i\gamma_k}\bar{\epsilon}(x)dx=\Re\int \nabla Q_{\omega_k}(\cdot-x_k)e^{i\gamma_k}\bar{\epsilon}(x)dx=0.
\end{align*}
Moreover, if $u\in B(\alpha,\sigma_1)$ for $0<\alpha<\alpha_1$ and $\sigma>\sigma_1$,  then
\begin{align}\label{decom:3}
    \|\epsilon\|_{H^\frac{1}{2}}+\sum_{k=1}^2|\omega_k-\omega_k^0|\leq C_1\alpha,\\\label{decom:trans}
    \min\{|x_1-x_2|\}\geq \sigma.
\end{align}
\end{lemma}
\begin{proof}
The proof is a standard application of the Implicit function theorem. For the reader's convenience, we give the proof of it.  Let $\alpha,\sigma>0$ and $x_1^0,,x_2^0$ be such that \eqref{def:sig0} holds and $\gamma_1^0,\gamma_2^0\in\mathbb{R}$. Denote by $B_0$ the $H^\frac{1}{2}$-ball of the center $\sum_{k=1}^2Q_{\omega_k^0}(\cdot-x_k^0)e^{i\gamma_k^0}$ and of radius $10\alpha$. For any $u\in B_0$ and the parameters $\omega_1,\omega_2$, $x_1,x_2$, $\gamma_1,\gamma_2$, Let $\rho=(\omega_1,\omega_2, x_1,x_2, \gamma_1,\gamma_2,u)$, and define
\begin{align*}
    \epsilon_{\rho}(x)=u(x)-\sum_{k=1}^2Q_{\omega_k}(\cdot-x_k)e^{i\gamma_k}.
\end{align*}
Now we define the functions of $\rho$, for $k=1,2$,
\begin{align*}
    \eta_k^1(\rho)=&\Re\int Q_{\omega_k}(x-x_k)e^{i\gamma_k}\bar{\epsilon}_{\rho}(x)dx,\\
    \eta_k^2(\rho)=&\Re\int Q^{\prime}_{\omega_k}(x-x_k)e^{i\gamma_k}\bar{\epsilon}_{\rho}(x)dx,\\
    \eta_k^3(\rho)=&\Im\int Q_{\omega_k}(x-x_k)e^{i\gamma_k}\bar{\epsilon}_{\rho}(x)dx,
\end{align*}
for $\rho$ close to
\begin{align*}
    \rho_0=\left(\omega_1^0,\omega_2^0, x_1^0,x_2^0, \gamma_1^0,\gamma_2^0,\sum_{k=1}^2Q_{\omega_k^0}(\cdot-x_k^0)e^{i\gamma_k^0}\right).
\end{align*}
For $\rho=\rho_0$, we have $\epsilon_{\rho_0}(x)\equiv0$, and thus $\eta^j_{k}(\rho_0)=0$, for $j=1,2,3$ and $k=1,2$. We check by applying the implicit function theorem that for any $u\in B_0$, one can choose in a unique way the coefficients $(\omega_1,\omega_2, x_1,x_2, \gamma_1,\gamma_2)$, so that $\rho$ is close to $\rho_0$ and verifies $\eta_k^j(\rho)=0$ for $j=1,2,3$. We compute the derivatives of $\eta_k^j$ for any $k$ and $j$ with respect to each $(\omega_k,x_k,\gamma_k)$. Note that
\begin{align*}
    \frac{\partial\epsilon_\rho}{\partial\omega_k}(\rho_0)=&-\frac{\partial Q_{\omega}}{\partial\omega}(\cdot-x_k^0)e^{i\gamma_k^0}\Big|_{\omega=\omega_k^0},\\
    \frac{\partial\epsilon_\rho}{\partial x_k}(\rho_0)=&Q^{\prime}_{\omega_k^0}(\cdot-x_k^0)e^{i\gamma_k^0},\\
    \frac{\partial\epsilon_\rho}{\partial\gamma_k}(\rho_0)=&-iQ_{\omega_k^0}(\cdot-x_k^0)e^{i\gamma_k^0}.
\end{align*}
Thus
\begin{align*}
    \frac{\partial\eta_{k'}^1}{\partial \omega_k}(\rho_0)=&-\Re\int Q_{\omega_{k'}^0}(x-x_{k'}^0)e^{i\gamma_{k'}^0}\frac{\partial Q_{\omega}}{\partial\omega}(\cdot-x_k^0)\Big|_{\omega=\omega_k^0}e^{-i\gamma_k^0}dx,\\
     \frac{\partial\eta_{k'}^1}{\partial x_k}(\rho_0)=&\Re\int  Q_{\omega_{k'}^0}(x-x_{k'}^0)e^{i\gamma_{k'}^0}Q^{\prime}_{\omega_k^0}(\cdot-x_k^0)e^{-i\gamma_k^0}dx,\\
      \frac{\partial\eta_{k'}^1}{\partial \gamma_k}(\rho_0)=&-\Im\int  Q_{\omega_{k'}^0}(x-x_{k'}^0)e^{i\gamma_{k'}^0}Q_{\omega_k^0}(\cdot-x_k^0)e^{-i\gamma_k^0}dx,
\end{align*}
and similar formulas hold for $\frac{\partial\eta_{k'}^j}{\partial \omega_k}(\rho_0)$, $\frac{\partial\eta_{k'}^j}{\partial x_k}(\rho_0)$, $\frac{\partial\eta_{k'}^j}{\partial \gamma_k}(\rho_0)$, for $j=2,3$.

Now we finish the computations for $k'=k$. The assumption \eqref{condition} implies that $\frac{\partial\eta_{k}^1}{\partial \omega_k}(\rho_0)<0$; Since $Q_{\omega_k}$ is even, $\frac{\partial\eta_{k}^1}{\partial x_k}(\rho_0)=0$; and finally, since $Q_{\omega_k}$ is real, $\frac{\partial\eta_{k}^1}{\partial \gamma_k}(\rho_0)=0$. The same applies for $\eta_k^j$, $j=2,3$. Hence, we obtain the following:
\begin{align}\label{D:2}
    \frac{\partial\eta_{k}^1}{\partial \omega_k}(\rho_0)<0,~~\frac{\partial\eta_{k}^1}{\partial x_k}(\rho_0)=0,~~\frac{\partial\eta_{k}^1}{\partial \gamma_k}(\rho_0)=0,\notag\\
    \frac{\partial\eta_{k}^2}{\partial \omega_k}(\rho_0)=0,~~\frac{\partial\eta_{k}^2}{\partial x_k}(\rho_0)>0,~~\frac{\partial\eta_{k}^2}{\partial \gamma_k}(\rho_0)=0,\notag\\
    \frac{\partial\eta_{k}^3}{\partial \omega_k}(\rho_0)=0,~~\frac{\partial\eta_{k}^3}{\partial x_k}(\rho_0)=0,~~\frac{\partial\eta_{k}^3}{\partial \gamma_k}(\rho_0)>0.
\end{align}
For $k'\neq k$, and $j=1,2,3$, since the different $Q_{\omega_k}$ are algebraic decaying and located at centers distant at least of $\sigma$, we have
\begin{align}\label{D:12}
    \left|\frac{\partial\eta_{k'}^j}{\partial\omega_k}(\rho_0)\right|+\left|\frac{\partial\eta_{k'}^j}{\partial\omega_k}(\rho_0)\right|+\left|\frac{\partial\eta_{k'}^j}{\partial\omega_k}(\rho_0)\right|\leq\langle \sigma\rangle^{-2}.
\end{align}
Theses terms are arbitrarily small by choosing $\sigma$ large.

Therefore, by \eqref{D:2} and \eqref{D:12}, the Jacobian of $\eta=(\eta_1^1,\eta_2^1,\eta_1^2,\eta_2^2,\eta_1^3,\eta_2^3)$ as a function of $\rho$ at the point $\rho_0$ is not zero. Thus we can apply the implicit function theorem to prove that for $\alpha$ small and $u\in B_0$, the existence and uniqueness of parameters $\rho$ such that $\eta(\rho)=0$. We obtain an estimate \eqref{decom:3} with constants that are independent of the ball $B_0$. This proves the result for $u\in B_0$. If we now take $u\in B(\alpha,\sigma_0)$, then $u$ belongs to such a ball $B_0$, and the result follows.
\end{proof}

As a consequence of this decomposition, we have the following result.
\begin{corollary}\label{corollary:1}
There exist $\alpha_1,\sigma_1,C_1>0$ such that the following is true. If for  $\sigma>\sigma_1$, $0<\alpha<\alpha_1$ and $t_0>0$,
\begin{align}\label{decom:10}
    u(t)\in B(\alpha,\sigma)~~\text{for all}~~t\in[0,t_0],
\end{align}
then there exist unique $C^1$-function $\omega_k:[0,t_0]\to(0,+\infty)$, $x_k,\gamma_k:[0,t_0]\to\mathbb{R}$ such that if we set
\begin{align*}
    \epsilon(t,x)=u(t,x)-R(t,x),
\end{align*}
where
\begin{align*}
    R(t,x)=\sum_{k=1}^2R_k(t,x),~~R_k(t,x)=Q_{\omega_k(t)}(x-x_k(t))e^{i\gamma_k(t)},
\end{align*}
then $\epsilon(t)$ satisfies, for all $k=1,2$ and all $t\in [0,t^*]$,
\begin{align}\label{decom:12}
    \Re\int R_k(t)\bar{\epsilon}(t)=\Im\int R_k(t)\bar{\epsilon}(t)=\Re\int \partial_xR_k(t)\bar{\epsilon}(t)=0.
\end{align}
Moreover, for all $t\in [0,t^*]$ and for all $k=1,2$,
\begin{align}\label{decom:13}
    &\|\epsilon(t)\|_{H^\frac{1}{2}}+\sum_{k=1}^2|\omega(t)-\omega_k^0|\leq C_1\alpha,\\\label{decom:14}
    &\min\{|x_1-x_2|\}\geq \sigma\\ \label{decom:15}
    &|\omega^{\prime}_k(t)|+|x_k^{\prime}(t)|^2+|\gamma_k^{\prime}(t)-\omega_k(t)|^2\leq C\int \langle x-x_k\rangle^{-2}\epsilon^2(t,x)dx+C_1\langle \sigma\rangle^{-2}.
\end{align}
\end{corollary}
\begin{proof}
Assume that $u(t)$ satisfies \eqref{decom:10} on $[0,t_0]$. Then applying Lemma \ref{lemma:decom:unique} to $u(t)$ for any $t\in [0,T_0]$, since the map $t\mapsto u(t)$ is continuous in $H^{s}$, $s\in(\frac{1}{2},1)$, we obtain for any $k=1,2$ the existence of continuous functions $\omega_k:[0,t_0]\to(0,\infty)$, such that \eqref{decom:12} holds. Moreover, \eqref{decom:13} and \eqref{decom:14} are the consequence of \eqref{decom:3}  and \eqref{decom:trans}, respectively. To prove that the functions $\omega_k$, $x_k$ and $\gamma_k$ are in fact of class $C^1$, we use a regularization argument and computations based on the equation of $\epsilon(t)$. These computations also justify estimates \eqref{decom:15}. The similar argument can be found in \cite[Corollary 3]{MMT2006Duke} or \cite{MM2001GAFA}. Now we give the equation of $\epsilon(t)$ to justify formally estimate \eqref{decom:15}. It is straightforward to check that the equation of $\epsilon(t)$ is
\begin{align}\label{equ:eps}
    i\partial_t\epsilon+\mathcal{L}_K\epsilon=&-i\sum_{k=1}^2\omega^{\prime}_k(t)\frac{\partial Q_\omega}{\partial \omega}|_{\omega_k(t)}(x-x_k(t))e^{i\gamma_k(t)}\notag\\
     &+i\sum_{k=1}^2x^{\prime}_k(t)Q^{\prime}_{\omega_k(t)}(x-x_k(t))e^{i\gamma_k(t)}\notag\\
    &+\sum_{k=1}^2(\gamma^{\prime}_k(t)-\omega_k(t))Q{\omega_k(t)}(x-x_k(t))e^{i\gamma_k(t)}\notag\\
    &+\mathcal{O}(\|\epsilon\|_{H^1}^2)+\mathcal{O}(\sigma^{-2}),
\end{align}
where
\begin{align*}
    \mathcal{L}_K\epsilon=D\epsilon-\sum_{k=1}^2\left(|Q_{\omega_k(t)}|^{p-1}\epsilon+(p-1)|Q_{\omega_k(t)}|^{p-2}\Re(Q_{\omega_k(t)}(x-x_k(t))e^{i\gamma_k(t)}\epsilon)\right).
\end{align*}
From the equation \eqref{equ:eps} of $\epsilon$, it is straightforward by taking scalar products by $Q_{\omega_k(t)}$ and then by $Q^{\prime}_{\omega_k(t)}$, we can obtain the estimate of  $|\omega_k^{\prime}(t)|^2$, $|x_k^{\prime}(t)|^2$, $|\gamma^{\prime}_k(t)-\omega_k(t)|^2$.
\end{proof}

%  {\color{red} \rule[-10pt]{16cm}{0.3em}}

\section{Weak stability of a single solitary wave}\label{sec:s}
 Let $u_0\in H^s(\mathbb{R})$, $s\in(\frac{1}{2},1)$, and for some $x_0\in\mathbb{R}$ and $\gamma_0\in \mathbb{R}$,
 \begin{align*}
     \|u_0-Q_{\omega_0}(x-x_0)e^{i\gamma_0}\|_{H^\frac{1}{2}}\leq\alpha,
 \end{align*}
 for $\alpha>0$ small enough. Let $u(t)$ be the corresponding solution of \eqref{equ-hf} with the initial data $u_0\in H^s$.

{ \bf Decomposition of the solution.} We argue on a time interval $[0,t^*]$, so that for all $t\in[0,t^*]$, $u(t)$ is close to $Q_{\omega(t)}(x-x(t))e^{i\gamma(t)}$ for some $\omega(t)$, $x(t)$ and $\gamma(t)$ in $H^s$. We can modify the parameters $\omega(t),x(t)$ and $\gamma(t)$ such that
\begin{align}\label{def:epsilon}
    \epsilon(t,x)=u(t,x)-R_0(t,x),
\end{align}
where
\begin{align}\label{def:R0}
    R_0(t,x)=Q_{\omega(t)}(x-x(t))e^{i\gamma(t)}
\end{align}
satisfies the orthogonal conditions
\begin{align}\label{ortho:condition}
    \Re(\epsilon(t),R_0(t))=\Im(\epsilon(t),R_0(t))=\Re(\epsilon(t), R_0^{\prime}(t))=0.
\end{align}
This choice of orthogonality conditions are well adapted to the positivity properties on the operators $L_\omega^+$ and $L_\omega^-$ (see Lemma \ref{lemma:linear:operator}), and thus it is suitable to apply an energy method.

Note that, as in Lemma \ref{lemma:decom:unique}, we have
\begin{align}\label{D:3}
    \|\epsilon(0)\|_{H^\frac{1}{2}}+|\omega(0)-\omega_0|\leq C\alpha.
\end{align}
By expanding $u(t)=R_0(t)+\epsilon(t)$ in the definition of $G_\omega$ (see \eqref{def:functional}), we obtain the following formula.
\begin{lemma}\label{lemma:J1}
The following holds:
\begin{align*}
    G_{\omega(0)}(u(t))=G_{\omega(0)}(Q_{\omega(0)})+H_0(\epsilon,\epsilon)+\|\epsilon(t)\|_{H^\frac{1}{2}}^2\beta\left(\|\epsilon(t)\|_{H^\frac{1}{2}}\right)+\mathcal{O}(|\omega(t)-\omega(0)|^2),
\end{align*}
with $\beta(\epsilon)\to0$ as $\epsilon\to0$, where
\begin{align}\label{def:H0}
   H_0(\epsilon,\epsilon)= \frac{1}{2}\int|D^{\frac{1}{2}}\epsilon|^2+\frac{\omega(t)}{2}\int|\epsilon|^2-\int \left(\frac{1}{2}|R_0|^{p-1}|\epsilon|^2+\frac{p-1}{2}|R_0|^{p-3}(\Re(R_0\epsilon))^2\right).
\end{align}
\end{lemma}
\begin{proof}
First, we consider the term $\int|D^{\frac{1}{2}}u|^2$:
\begin{align*}
\int|D^{\frac{1}{2}}u|^2=&\int(D(R_0+\epsilon),R_0+\epsilon)\\
=&\int|D^{\frac{1}{2}}R_0|^2+2\Re\int D\bar{R}_0\epsilon+\int|D^{\frac{1}{2}}\epsilon|^2.
\end{align*}
For the nonlinear term, we have
\begin{align*}
\int|u|^{p+1}=\int\left(|R_0+\epsilon|^2\right)^{\frac{p+1}{2}}
=&\int\Big[|R_0|^{p+1}+(p+1)|R_0|^{p-1}\Re(R_0\epsilon)+\frac{p+1}{2}|R_0|^{p-1}|\epsilon|^2\\&+\frac{p^2-1}{2}|R_0|^{p-3}(\Re(R_0\epsilon))^2\Big]+\|\epsilon(t)\|_{H^\frac{1}{2}}^2\beta\left(\|\epsilon(t)\|_{H^\frac{1}{2}}\right).
\end{align*}
For the $L^2$ term, we have
\begin{align*}  \int|u|^2=\int|R_0|^2+2\Re\int\bar{R}_0\epsilon+\int|\epsilon|^2.
\end{align*}
Combining above equalities, we get
\begin{align*}
    G_{\omega(0)}(u(t))=&G_{\omega(0)}(R_0(t))+\Re\int D\bar{R}_0\epsilon+\frac{1}{2}\int|D^{\frac{1}{2}}\epsilon|^2+\frac{\omega(0)}{2}\left(2\Re\int\bar{R}_0\epsilon+\int|\epsilon|^2\right)\\\
    &-\int \left(|R_0|^{p-1}\Re(R_0\epsilon)+\frac{1}{2}|R_0|^{p-1}|\epsilon|^2+\frac{p-1}{2}|R_0|^{p-3}(\Re(R_0\epsilon))^2\right)\\
    &+\|\epsilon(t)\|_{H^\frac{1}{2}}^2\beta\left(\|\epsilon(t)\|_{H^\frac{1}{2}}\right)\\
    =&G_{\omega(0)}(R_0(t))+H_0(\epsilon,\epsilon)+\Re\int D\bar{R}_0\epsilon+\omega(0)\Re\int\bar{R}_0\epsilon-\int |R_0|^{p-1}\Re(R_0\epsilon)\\
    &+\frac{\omega(0)-\omega(t)}{2}\int|\epsilon|^2+\|\epsilon(t)\|_{H^\frac{1}{2}}^2\beta\left(\|\epsilon(t)\|_{H^\frac{1}{2}}\right)\\
    =&G_{\omega(0)}(R_0(t))+H_0(\epsilon,\epsilon)+I+\frac{\omega(0)-\omega(t)}{2}\int|\epsilon|^2+\|\epsilon(t)\|_{H^\frac{1}{2}}^2\beta\left(\|\epsilon(t)\|_{H^\frac{1}{2}}\right),
\end{align*}
where $H_0(\epsilon,\epsilon)$ is defined by \eqref{def:H0} and
\begin{align*}
    I:=\Re\int D\bar{R}_0\epsilon+\omega(0)\Re\int\bar{R}_0\epsilon-\int |R_0|^{p-1}\Re(R_0\epsilon).
\end{align*}
By the definition of $R_0$ (see \eqref{def:R0}) and the equality $DQ_{\omega}+\omega Q_{\omega}= |Q_{\omega}|^{p-1}Q_{\omega}$, we have
\begin{align*}
    I:=&\Re\int \left(D\bar{R}_0\epsilon+\omega(t)\bar{R}_0\epsilon- |R_0|^{p-1}R_0\epsilon\right)+(\omega(0)-\omega(t))\Re\int \bar{R}_0\epsilon\\
    =&\Re\int \left(DQ_{\omega(t)}+\omega(t)Q_{\omega(t)}- |Q_{\omega(t)}|^{p-1}Q_{\omega(t)}\right)e^{-i\gamma
    (t)}\epsilon+(\omega(0)-\omega(t))\Re\int \bar{R}_0\epsilon\\
    =&(\omega(0)-\omega(t))\Re\int \bar{R}_0\epsilon\\
    =&0,
\end{align*}
where in the last step we used the fact that the orthogonality condition \eqref{ortho:condition}.

On the other hand, we have
\begin{align*}
    \frac{|\omega(0)-\omega(t)|}{2}\int|\epsilon|^2\leq\frac{1}{4}|\omega(0)-\omega(t)|^2+\frac{1}{4}\left(\int|\epsilon|^2\right)^2.
\end{align*}
Combining the above estimates, we can obtain the desire result.
\end{proof}

Next, we give the following positivity of the quadratic form $H_0$.
\begin{lemma}\label{lemma:coer:2}
There exists $\lambda_0>0$ such that if $\epsilon(t)\in H^{\frac{1}{2}}(\mathbb{R})$ satisfies
\begin{align*}
    \Re(\epsilon(t),R_0(t))= \Im(\epsilon(t),R_0(t))= \Re(\epsilon(t), R_0^{\prime}(t))=0,
\end{align*}
then
\begin{align*}
    H_0(\epsilon(t),\epsilon(t))\geq\lambda_0\|\epsilon(t)\|_{H^{\frac{1}{2}}}^2,
\end{align*}
where $H_0$ is given by \eqref{def:H0}.
\end{lemma}
\begin{proof}
This lemma is a direct consequence of the following claim applied to $Q_{\omega(t)}$ and $\epsilon$.

{\bf Claim:} Let $\omega_0>0$, $x_0\in\mathbb{R}$, $\gamma_0\in\mathbb{R}$  and $Q_{\omega_0}$ is the solution of \eqref{equ-elliptic}. Now we consider the quadratic form
\begin{align*}
    \Tilde{H}_0(\eta,\eta)= &\frac{1}{2}\int|D^{\frac{1}{2}}\eta|^2+\frac{\omega(0)}{2}\int|\eta |^2\\
    &-\int \left(\frac{1}{2}|Q_{\omega_0}(\cdot-x_0)|^{p-1}|\eta|^2+\frac{p-1}{2}|Q_{\omega_0}(\cdot-x_0)|^{p-3}(\Re(Q_{\omega_0}(\cdot-x_0)\eta))^2\right).
\end{align*}
There exists $\lambda_1>0$ such that if $\eta=\eta_1+i\eta_2\in H^{\frac{1}{2}}(\mathbb{R})$ satisfies
\begin{align*}
    \Re\int Q_{\omega_0}(\cdot-x_0)e^{-i\gamma_0}\eta=\Re\int Q_{\omega_0}^{\prime}(\cdot-x_0)e^{-i\gamma_0}\eta=\Im\int Q_{\omega_0}(\cdot-x_0)e^{-i\gamma_0}\eta=0.
\end{align*}
Then
\begin{align*}
    \Tilde{H}_0(\eta,\eta)\geq\lambda_1\|\eta\|_{H^{\frac{1}{2}}}^2.
\end{align*}
Indeed, we have
\begin{align*}
    \Tilde{H}_0(\eta,\eta)=(L_{\omega_0}^+\eta_1,\eta_1)+(L_{\omega_0}^-\eta_2,\eta_2)\geq\lambda_1\|\eta\|_{H^{\frac{1}{2}}}^2,
\end{align*}
where in the last step, we used  the orthogonality conditions and Lemma \ref{lemma:linear:operator}. Then the claim is true and we complete the proof of this lemma.
\end{proof}
    The following lemma aims to control the $H^{\frac{1}{2}}$ norm of $\epsilon(t)$.
\begin{lemma}
Assume that $\epsilon(t)$ is given by \eqref{def:epsilon}. Then we have the following estimate:
\begin{align}\label{D:5}
  \|\epsilon(t)\|_{H^\frac{1}{2}}^2\leq C|\omega(t)-\omega(0)|^2+C\|\epsilon(0)\|_{H^\frac{1}{2}}^2 .
\end{align}
if $\|\epsilon(0)\|_{H^\frac{1}{2}}$ is small enough.
\end{lemma}
\begin{proof}
   Since $G_{\omega(0)}(u(t))$ is the sum of two conserved quantities, we have
\begin{align*}
    G_{\omega(0)}(u(t))=G_{\omega(0)}(u(0)).
\end{align*}
Thus, from Lemma \ref{lemma:J1}, it follows that
\begin{align*}
    H_0(\epsilon(t),\epsilon(t))\leq&  H_0(\epsilon(0),\epsilon(0))+C|\omega(t)-\omega(0)|^2\\
    &+C\|\epsilon(0)\|_{H^\frac{1}{2}}^2\beta\left(\|\epsilon(0)\|_{H^\frac{1}{2}}\right)+C\|\epsilon(t)\|_{H^\frac{1}{2}}^2\beta\left(\|\epsilon(t)\|_{H^\frac{1}{2}}\right).
\end{align*}
By Lemma \ref{lemma:coer:2} and since $H_0(\epsilon(0),\epsilon(0))\leq C\|\epsilon(0)\|_{H^\frac{1}{2}}^2$, we obtain
\begin{align*}
    \lambda_0\|\epsilon(t)\|_{H^\frac{1}{2}}^2\leq H_0(\epsilon(t),\epsilon(t))\leq C|\omega(t)-\omega(0)|^2+C\|\epsilon(0)\|_{H^\frac{1}{2}}^2+C\|\epsilon(t)\|_{H^\frac{1}{2}}^2\beta\left(\|\epsilon(t)\|_{H^\frac{1}{2}}\right).
\end{align*}
Using Lemma \ref{lemma:J1} again and above estimates, we have
\begin{align*}
    \|\epsilon(t)\|_{H^\frac{1}{2}}^2\leq C|\omega(t)-\omega(0)|^2+C\|\epsilon(0)\|_{H^\frac{1}{2}}^2,
\end{align*}
for $ \|\epsilon(t)\|_{H^\frac{1}{2}}$ small enough. Now we complete the proof of this Lemma.
\end{proof}

Finally, we need to control the parameter $|\omega(t)-\omega(0)|$.
\begin{lemma}
Assume that $\omega(t)$ and $\omega(0)$ is given by above. Then the following holds.
\begin{align}\label{D:7}
     |\omega(t)-\omega(0)|\leq C\left(\|\epsilon(t)\|_{L^2}^2+\|\epsilon(0)\|_{L^2}^2\right).
\end{align}
\end{lemma}
\begin{proof}
    We prove that $|\omega(t)-\omega(0)|$ is quadratic in $\epsilon(t)$. Note that by the conservation of $\|u(t)\|_{L^2}$ and the orthogonality condition $\Re(R_0,\epsilon)=0$, we have
\begin{align}\label{D:6}
    \int Q_{\omega(t)}^2-\int Q_{\omega(0)}^2=-\int|\epsilon(t)|^2+\int|\epsilon(0)|^2.
\end{align}
Recall that we assume $\frac{d}{d\omega}\int Q_{\omega}^2(x)dx|_{\omega=\omega_0}>0$, and $\omega(t),\omega(0)$ are close to $\omega_0$. Thus,
\begin{align*}
    (\omega(t)-\omega(0))\left(\frac{d}{d\omega}\int Q_\omega^2(x)dx|_{\omega=\omega_0}\right)=\int Q_{\omega(t)}^2-\int Q_{\omega(0)}^2+\beta(\omega(t)-\omega(0))(\omega(t)-\omega(0))^2,
\end{align*}
with $\beta(\epsilon)\to0$ as $\epsilon\to0$. This implies that for some $C=C(\omega_0)>0$,
\begin{align*}
    |\omega(t)-\omega(0)|\leq C\left|\int Q_{\omega(t)}^2-\int Q_{\omega(0)}^2\right|.
\end{align*}
Therefore, by \eqref{D:6}, we can obtain the desired result.
\end{proof}

 \begin{proof}[\bf Proof of the stability of a single solitary wave.]
  Combining \eqref{D:3}, \eqref{D:5} and \eqref{D:7}, we have, for some constant $C>0$,
\begin{align*}
    \|\epsilon(t)\|_{H^{\frac{1}{2}}}^2+|\omega(t)-\omega(0)|\leq C\|\epsilon(0)\|_{H^\frac{1}{2}}^2\leq C\alpha,
\end{align*}
for $ \|\epsilon(t)\|_{H^\frac{1}{2}}$ and $|\omega(t)-\omega(0)|$ small enough. Thus, for $\alpha$ small enough,
\begin{align*}
    &\left\|u(t)-Q_{\omega_0}(x-x(t))e^{i\gamma(t)}\right\|_{H^\frac{1}{2}}^2\\
    \leq& \left\|u(t)-R_0(t)\right\|_{H^\frac{1}{2}}^2+\left\|R_0(t)-Q_{\omega_0}(x-x(t))e^{i\gamma(t)}\right\|_{H^\frac{1}{2}}^2\\
    \leq& \|\epsilon(t)\|_{H^\frac{1}{2}}^2+\|\epsilon(0)\|_{H^\frac{1}{2}}^2+C|\omega(t)-\omega_0|\\ %define of norm +using Pohozaev identity +L^2-norm is bounded
    \leq& \|\epsilon(t)\|_{H^\frac{1}{2}}^2+\|\epsilon(0)\|_{H^\frac{1}{2}}^2+|\omega(t)-\omega(0)|+C|\omega(0)-\omega_0|\\
    \leq& C\alpha.
\end{align*}
This complete the proof of stability of a single solitary wave.
 \end{proof}

% {\color{red} \rule[-10pt]{16cm}{0.3em}}

\section{Stability the sum of multi-solitary waves}
In this section, we aim to prove the stability of the multi-solitary waves.

 For $A_0, \sigma_0,\alpha>0$, we define
\begin{align*}
    \Gamma_{A_0}(\alpha,\sigma)=\left\{u\in H^s(\mathbb{R});\inf_k\left\|u(t,\cdot)-\sum_{k=1}^2Q_{\omega_k^0}(\cdot-y_k)e^{i\gamma_k}\right\|_{H^\frac{1}{2}}\leq A_0\left(\alpha+\langle \sigma\rangle^{-1}\right)\right\},
\end{align*}
where $\sigma$ is given by \eqref{def:sig0}.

Let $\omega_k^0$, $x_k^0$, $\gamma_k^0$ be defined as in the statement of Theorem \ref{Thm1}. We claim that Theorem \ref{Thm1} is a consequence of the following proposition.
\begin{proposition}\label{proposition:1}
{(Reduction of the problem)}
There exists $A_0>2$, $\sigma_0>1$, and $\alpha_0>0$ such that for all $u_0\in H^s(\mathbb{R})$, $s\in(\frac{1}{2},1)$, if
\begin{align}\label{ass:sta:1}
   \left\|u_0- \sum_{k=1}^2Q_{\omega_k^0}(\cdot-x_k^0)e^{i\gamma_k^0}\right\|_{H^\frac{1}{2}}\leq\alpha,
\end{align}
where $\sigma>\sigma_0$, $0<\alpha<\alpha_0$, and $x_k^0$ satisfy \eqref{def:sig0}, and if for some $t^*>1$,
\begin{align*}
    u(t)\in \Gamma_{A_0}(\alpha,\sigma),~~\text{for any }~~t\in[0,t^*].
\end{align*}
Then, for any $t\in[0,t^*]$,
\begin{align*}
    u(t)\in \Gamma_{A_0/2}(\alpha,\sigma).
\end{align*}
\end{proposition}
\begin{proof}[\bf Proof of Theorem \ref{Thm1}.]
Assuming that Proposition \ref{proposition:1} is true, we check that it implies Theorem \ref{Thm1}. In fact, suppose that $u_0$ satisfies the assumptions of Theorem \ref{Thm1}. Let $u(t)\in H^s(\mathbb{R})$, $s\in\left(\frac{1}{2},1\right)$, be the solution of \eqref{equ-hf} with the initial data $u_0\in H^s(\mathbb{R})$, $s\in\left(\frac{1}{2},1\right)$. Then, by the continuity of $u(t)$ in $H^s$, there exists $\tau>0$ such that for any $0\leq t\leq \tau$, $u(t)\in \Gamma_{A_0}(\alpha,\sigma)$. Let
\begin{align*}
    t^*=\sup\{t\geq0, u(t')\in \Gamma_{A_0}(\alpha,\sigma),~~ t'\in[0,t]\}.
\end{align*}
Assume for the sake of contradiction that $t^*$ is not $+\infty$, then by Proposition \ref{proposition:1}, for all $t\in[0,t^*]$, $u\in\Gamma_{A_0/2}(\alpha,\sigma) $. Since $u(t)$ is continuous in $H^s$, there exists $\tau'>0$ such that for all $t\in [0,t^*+\tau']$, $u(t)\in\Gamma_{2A_0/3}(\alpha,\sigma)$, which contradicts the definition of $t^*$. Therefore, $t^*=+\infty$, and \eqref{a:2} in Theorem \ref{Thm1} holds.
\end{proof}
 The rest of this section is to prove Proposition \ref{proposition:1}. We divided the proof into the following four steps.

 {\bf Step 1. Decomposition of the solution around two solitary waves.}

 First, since for all $t\in[0,t^*]$, $u\in \Gamma_{A_0}(\alpha,\sigma)$, by choosing $\sigma_0=\sigma_0(A_0)$ large enough and $\alpha_0=\alpha_0(A_0)>0$ small enough, we can apply the Corollary \ref{corollary:1} to $u(t)$ in the time interval $[0,t^*]$. It follows that there exist unique $C^1$-functions $\omega_k:[0,t^*]\to(0,+\infty)$, $x_k,\gamma_k:[0,t^*]\to\mathbb{R}$ such that if we set
\begin{align*}
    \epsilon(t,x)=u(t,x)-R(t,x),
\end{align*}
where
\begin{align}\label{def:Rk}
    R(t,x)=\sum_{k=1}^2R_k(t,x),~~R_k(t,x)=Q_{\omega_k(t)}(x-x_k(t))e^{i\gamma_k(t)}.
\end{align}
Then $\epsilon(t)$ satisfies, for all $k=1,2$ and all $t\in [0,t^*]$,
\begin{align}\label{decom:small}
    &\Re\int R_k(t)\bar{\epsilon}(t)=\Im\int R_k(t)\bar{\epsilon}(t)=\Re\int \partial_xR_k(t)\bar{\epsilon}(t)=0,\notag\\
    &\|\epsilon(t)\|_{H^\frac{1}{2}}+\sum_{k=1}^2|\omega_k(t)-\omega_k^0|+|\omega_k^{\prime}(t)|+|x_k^{\prime}(t)|+|\gamma^{\prime}_k(t)-\omega_k(t)|\leq C_1A_0\left(\alpha+\langle \sigma\rangle^{-1}\right).
\end{align}
Moreover, for all $t\in [0,t^*]$ and for all $k=1,2$, by the assumption \eqref{ass:sta:1} on $u_0$ and Lemma \ref{lemma:decom:unique} applied to $u_0$, we have
\begin{align}\label{decom:4}
    \|\epsilon(0)\|_{H^\frac{1}{2}}+\sum_{k=1}^2|\omega(0)-\omega_k^0|\leq C_1\alpha,
\end{align}
where $C_1$ does not depend on $A_0$.

{\bf Step 2. Local mass monotonicity property.}

First, we introduce the localization functions which will be frequently used in the construction.
Let $\{x_k\}_{k=1}^K$ be the two distinct points in Theorem \ref{Thm1}.

Without loss of generality, in the following, we assume that $x_1<0$ and $x_2>0$, $x_1+x_2=0$.

Let $\Phi: \R\rightarrow [0,1]$ be a smooth function such that $|\Phi^\prime(x)| \leq C \sigma^{-1}$ for some $C>0$, where $\sigma$ is defined as \eqref{def:sig0}, $\Phi(x)=0$ for $x\leq 1$ and $\Phi(x)=1$ for $x\geq 3$. Define
\begin{align*}
 \Phi_1(x)=1~~\text{and}~~   \Phi_2(x) :=\Phi\left(\frac{x-x_2}{(t+\sigma)^2}\right).
\end{align*}

We also introduce a functional adapted to the stability problem for two solitary waves. We define
\begin{align}\label{def:G}
    G(t)=E(u(t))+ \frac{1}{2}\mathcal{J}(t),
\end{align}
where $E(u(t))$ is given by \eqref{energy} and
\begin{align*}
    \mathcal{J}(t)=\sum_{k=1}^2\mathcal{J}_k(t),
\end{align*}
and
\begin{align}\label{def:JK}
    \mathcal{J}_k(t)=\omega_k(0)\int|u(t,x)|^2\Phi_k(x)dx.
\end{align}
\begin{lemma}\label{lemma:mass:loc}
Let $\mathcal{J}_k$ be defined as \eqref{def:JK}, then we have
\begin{align}\label{est:mass:local}
\mathcal{J}_2(t)-\mathcal{J}_2(0)\leq \frac{C}{\sigma}\sup_{0<\tau<t}\|\epsilon(\tau)\|_{L^2}^2+\frac{C}{\sigma}.
\end{align}
where $\sigma$ is defined as \eqref{def:sig0}.
\end{lemma}
\begin{proof}
By equation \eqref{equ-hf} and $u(t,x)=R(t,x)+\epsilon(t,x)$, where $R$ is given by \eqref{def:Rk}, we deduce that
\begin{align} \label{du2-bc}
 \frac{d}{dt}\mathcal{J}_2(t)=&\omega_1(0)\frac{d}{dt}\int |u|^2\Phi_2dx
 = 2\omega_1(0){\Im}\int \bar{u} D u\Phi_2 dx+2\omega_1(0)\int |u|^2\Phi_2^{\prime}\notag\\
 =& 2\omega_1(0){\Im}\int \bar{R} DR\Phi_2 dx+2\omega_1(0){\Im}\int (\bar{R} D\epsilon+\bar{\epsilon} DR)\Phi_2 dx\notag\\
 &+2\omega_1(0){\Im}\int \bar{\epsilon} D\epsilon\Phi_2 dx+2\omega_1(0)\int |u|^2\Phi_2^{\prime}\notag\\
 =&:I+II+III+IV.
 \end{align}
We shall estimate the above three terms separately. Due to the non-local operator $D$, the estimates below are more delicate. In fact, the integration representation of the operator $D$ and the point-wise decay property of the ground state are used to obtain the right decay orders. And the Calder\'{o}n's estimate in \cite{C1965PNA} is also used below.

{\bf $(i)$ Estimate of $I$.}
The first term on the right-hand side of \eqref{du2-bc} can be decomposed as
\begin{align}
    \label{dec-i}
I&=2\omega_1(0){\Im}\int \bar{R}_2 DR_2\Phi_2 dx+2\omega_1(0){\Im}\int (\bar{R}_1 DR_2+\bar{R}_2DR_1)\Phi_2dx+2\omega_1(0){\Im}\int \bar{R}_1 DR_1\Phi_2 dx\notag\\
&=:I_1+I_2+I_3.
\end{align}
Since $D$ is self-adjoint, then integrating by parts and using the $\Im (\bar{u}v)=-\Im(u\bar{v})$, we have
\begin{align*}
    2\Im\int \bar{u}Du\phi=\Im \int uD(\bar{u}\phi)-\Im\int uD\bar{u}\phi
\end{align*}
From the above equality and the identity \eqref{id-fra2}, we deduce that
\begin{align*}
   &|I_1|\\
   \leq &\left|{\Im}\int R_2 (D(\bar{R}_2\Phi_2)-D\bar{R}_2\Phi_2) dx\right|\\
=&C\left|{\Im}\int R_2(x)dx\int\frac{\bar{R}_2(x+y)(\Phi_2(x+y)-\Phi_2(x))-\bar{R}_2(x-y)(\Phi_2(x)-\Phi_2(x-y))}
{|y|^2}dy\right|\\
=&C\Bigg|{\Im}\int R_2(x)dx\int\frac{\bar{R}_2(x+y)\left[\Phi\left(\frac{x-x_2+y}{(t+\sigma)^2}\right)-\Phi\left(\frac{x-x_2}{(t+\sigma)^2}\right)\right]}
{|y|^2}\notag\\
&~~~~~-\frac{\bar{R}_2(x-y)\left[\Phi\left(\frac{x-x_2}{(t+\sigma)^2}\right)-\Phi\left(\frac{x-x_2-y}{(t+\sigma)^2}\right)\right]}
{|y|^2}dy\Bigg|\notag\\
\leq C& \sum\limits_{j=1}^4
   \Bigg|{\Im} \iint_{\Omega_{1,j}}  R_2(x) \frac{\bar{R}_2(x+y)[\left[\Phi\left(\frac{x-x_2+y}{(t+\sigma)^2}\right)-\Phi\left(\frac{x-x_2}{(t+\sigma)^2}\right)\right]}{|y|^2}dxdy\\
   &~~~~-{\Im} \iint_{\Omega_{1,j}}  R_2(x)\frac{\bar{R}_2(x-y)\left[\Phi\left(\frac{x-x_2}{(t+\sigma)^2}\right)-\Phi\left(\frac{x-x_2-y}{(t+\sigma)^2}\right)\right]}{|y|^2} dxdy\Bigg| \\
   \leq C& \sum\limits_{j=1}^4
   \Bigg|{\Im} \iint_{\Omega_{1,j}}  R_2(x) \frac{\bar{R}_2(x+y)[\left[\Phi\left(\frac{x-x_2+y}{(t+\sigma)^2}\right)-\Phi\left(\frac{x-x_2}{(t+\sigma)^2}\right)\right]}{|y|^2}dxdy\\
   &~~~~-{\Im} \iint_{\Omega_{1,j}}  R_2(x)\frac{\bar{R}_2(x-y)\left[\Phi\left(\frac{x-x_2}{(t+\sigma)^2}\right)-\Phi\left(\frac{x-x_2-y}{(t+\sigma)^2}\right)\right]}{|y|^2} dxdy\Bigg| \\
=&:C \sum\limits_{j=1}^4 I_{1,j},
\end{align*}
where $\mathbb{R}^2$ is partitioned into four regimes
\begin{align*}
    \Omega_{1,1}:&=\left\{\left|\frac{x-x_2}{(t+\sigma)^2}\right|\leq 3,~ \left|\frac{y}{(t+\sigma)^2}\right|\leq 1\right\},\\
    \Omega_{1,2}:&=\left\{\left|\frac{x-x_2}{(t+\sigma)^2}\right|\leq 3,~ \left|\frac{y}{(t+\sigma)^2}\right|> 1\right\},\\
    \Omega_{1,3}:&=\left\{\left|\frac{x-x_2}{(t+\sigma)^2}\right|>3,~ \left|\frac{y}{(t+\sigma)^2}\right|\leq 1\right\},\\
    \Omega_{1,4}:&=\left\{\left|\frac{x-x_2}{(t+\sigma)^2}\right|> 3,~ \left|\frac{y}{(t+\sigma)^2}\right|> 1\right\}.
\end{align*}

{\bf Now we estimate $I_{1,j}$, $1\leq j\leq 4$.}

For the term $I_{1,1}$. By Taylor's expansion, there exists some $|\theta|\leq 1$ such that for any differentiable function $f$, we have
\begin{align}
    \label{Taylor}
&|f(x+y)(\Phi_2(x+y)-\Phi_2(x))-f(x-y)(\Phi_2(x)-\Phi_2(x-y))|\notag\\
&\leq |y|^2(|\nabla f (x+\theta y)| \|\nabla\Phi_2\|_{L^\infty}+|f(x)|\|\nabla^2\Phi_2\|_{L^\infty})\notag\\
&\leq C\frac{|y|^2}{(t+\sigma)^2}|(|\nabla f(x+\theta y)|+|f(x-y)|).
\end{align}
Thus, from \eqref{Taylor}, we get
\begin{align*}
    |I_{1,1}|\leq &\frac{C}{(t+\sigma)^2}\left|\int_{1\leq\frac{x-x_2}{(t+\sigma)^2}<3} Q_{\omega_2}(x-x_2) \int_{\left|\frac{y}{(t+\sigma)^2}\right|<1}(|\nabla Q_{\omega_2}(x+\theta y-x_2)|+|Q_{\omega_2}(x-x_2-y)|)dydx\right|
\end{align*}
Since $Q\in C^1(\mathbb{R})$ and bounded , then we have
\begin{align}
    I_{1,1}
    \leq C\frac{1}{(t+\sigma)^2}.
\end{align}
For the term $I_{1,2}$. Thus, from the definition of $\Phi$ and the bounded of $Q_{\omega_2}$, we get
\begin{align*}
    |I_{1,2}|\leq &\frac{C}{(t+\sigma)^2}\left|\int_{\left|\frac{x-x_2}{(t+\sigma)^2}\right|<3} Q_{\omega_2}(x-x_2) \int_{\left|\frac{y}{(t+\sigma)^2}\right|\geq1}\frac{Q_{\omega_2}(x-y-x_2)+Q_{\omega_2}(x+y-x_2)}{|y|^2}dydx\right|\\
    \leq&\frac{C}{(t+\sigma)^2}\left|\int_{\left|\frac{x-x_2}{(t+\sigma)^2}\right|<3} Q_{\omega_2}(x-x_2) \int_{\left|\frac{y}{(t+\sigma)^2}\right|\geq1}\frac{1}{|y|^2}dydx\right|\\
    \leq&\frac{C}{(t+\sigma)^4}.
\end{align*}
For the term $I_{1,3}$, by the definition of
 $\Phi$, we deduce $I_{1,3}=0$.

For the term $I_{1,4}$. Since $(x,y)\in \Omega_{1,4}$, by the definition of \eqref{def:sig0}, we have
\begin{align*}
    |I_{1,4}|\leq &\frac{C}{(t+\sigma)^2}\left|\int_{\left|\frac{x-x_2}{(t+\sigma)^2}\right|>3} Q_{\omega_2}(x-x_2) \int_{\left|\frac{y}{(t+\sigma)^2}\right|\geq1}\frac{Q_{\omega_2}(x-y-x_2)+Q_{\omega_2}(x+y-x_1)}{|y|^2}dydx\right|\\
    \leq&\frac{C}{(t+\sigma)^2}\left|\int_{\left|\frac{x-x_2}{(t+\sigma)^2}\right|>3} Q_{\omega_2}(x-x_2) \int_{\left|\frac{y}{(t+\sigma)^2}\right|\geq1}\frac{1}{|y|^2}dydx\right|\\
    \leq&\frac{C}{(t+\sigma)^4}.
\end{align*}
Hence, we conclude that
\begin{align}\label{est:i1}
|I_1|\leq C(t+\sigma)^{-2}.
\end{align}

{\bf Estimate of $I_2$.} The second term $I_2$ in \eqref{dec-i} can be estimated by renormalization \eqref{def:Rk}. In fact, by the definition of $\Phi_2$, we get
\begin{align*}
    &\Im\int (\bar{R}_1 DR_2+\bar{R}_2DR_1)\Phi_2 dx\notag\\
    \leq&\Im\int_{\frac{x-x_2}{(t+\sigma)^2}\geq1} (\bar{R}_1 DR_2+\bar{R}_2DR_1)\\
    =&\Im\int_{\frac{x-x_2}{(t+\sigma)^2}\geq1} (Q_{\omega_1}(x-x_1) DQ_{\omega_2}(x-x_2)+Q_{\omega_2}(x-x_2)DQ_{\omega_1}(x-x_1))\notag\\
    \leq&C\Im\int_{\frac{x-x_2}{(t+\sigma)^2}\geq1} \frac{1}{|x-x_2|^2}+\frac{1}{|x-x_2|^3}dx\\
    \leq&C(t+\sigma)^{-2}.
\end{align*}
where we have used the definition of $\Phi$ and $Q_{\omega_k}$ is uniformly bounded.

Hence
\begin{align}\label{est:i2}
    |I_2|\leq C (t+\sigma)^{-2}.
\end{align}

{\bf  Estimate of $I_3$.}
From \eqref{def:Rk}, we have
\begin{align}\label{est:i3}
   |I_3|=&\left| 2\omega_2(0)\Im\int \bar{R}_{1}DR_{1}\Phi_2(x)dx\right|\notag\\
   \leq&C\int_{\frac{x-x_2}{(t+\sigma)^2}\geq1} Q_{\omega_2}(x-x_1)DQ_{\omega_2}(x-x_1)dx\notag\\
   \leq&C\int_{\frac{x-x_2}{(t+\sigma)^2}\geq1}\frac{1}{|x-x_1|^2}dx\notag\\
   \leq&C\frac{1}{(t+\sigma)^2+x_2-x_1}\notag\\
   \leq&C\frac{1}{(t+\sigma)^2}.
\end{align}
where we have used the fact that $x_2-x_1>0$.

Thus, from \eqref{est:i1}, \eqref{est:i2} and \eqref{est:i3}, we deduce
\begin{align}\label{est-i}
    |I|\leq C(t+\sigma)^{-2}.
\end{align}

{\bf $(ii)$ Estimate of $II$.}
Regarding the second term $II$ on the right-hand side of \eqref{du2-bc},
we first apply the integration by parts formula and $\Im(\bar{u}v)=-\Im(u\bar{v})$ to get
\begin{align*}
{\Im}\int (\bar{R} D\epsilon+\bar{\epsilon} DR)\Phi_2 dx=&\sum_{k=1}^{2}{\Im}\int (\bar{R}_k D\epsilon+\bar{\epsilon} DR_k)\Phi_2 dx \\
=&\sum_{k=1}^{2}{\Im}\int \epsilon(D(\bar{R}_k\Phi_2)-D\bar{R}_k\Phi_2) dx
=: \sum\limits_{k=1}^2 II_k.
\end{align*}
Then, by \eqref{id-fra2} in Lemma \ref{lem-formula}, for $1\leq k\leq 2$,
\begin{align*}
   |II_k|=&C\left|{\Im}\int \epsilon(x)dx\int\frac{\bar{R}_k(x+y)(\Phi_2(x+y)-\Phi_2(x))-\bar{R}_k(x-y)(\Phi_2(x)-\Phi_2(x-y))}{|y|^2}dy\right|\\
   =&C\bigg|{\Im}\int \epsilon(x)dx\bigg(\int\frac{\bar{R}_k(x+y)\left[\Phi\left(\frac{x-x_2+y}{(t+\sigma)^2}\right)-\Phi\left(\frac{x-x_2}{(t+\sigma)^2}\right)\right]}{|y|^2}dy\\
   &-\int\frac{\bar{R}_k(x-y)\left[\Phi\left(\frac{x-x_2}{(t+\sigma)^2}\right)-\Phi\left(\frac{x-x_2-y}{(t+\sigma)^2}\right)\right]}{|y|^2}dy\bigg)\bigg|\\
\leq& C \sum\limits_{j=1}^3
     \Bigg|{\Im}\iint_{\Omega_{2,j}}  \epsilon(x)  \frac{\bar{R}_k(x+y)\left[\Phi\left(\frac{x-x_2+y}{(t+\sigma)^2}\right)-\Phi\left(\frac{x-x_2}{(t+\sigma)^2}\right)\right]}{|y|^2}\\
     &~~~~-{\Im}\iint_{\Omega_{2,j}}  \epsilon(x)\frac{\bar{R}_k(x-y)\left[\Phi\left(\frac{x-x_1}{(t+\sigma)^2}\right)-\Phi\left(\frac{x-x_1-y}{(t+\sigma)^2}\right)\right]}{|y|^2}  dxdy\Bigg| \\
=&: C \sum\limits_{j=1}^3 II_{k,j},
\end{align*}
where
\begin{align*}
    &\Omega_{2,1} := \left\{\left|\frac{x-x_2}{(t+\sigma)^2}\right|\leq 3,~ \left|\frac{y}{(t+\sigma)^2}\right|\leq 1\right\},\\
    &\Omega_{2,2} := \left\{\left|\frac{x-x_2}{(t+\sigma)^2}\right|> 3,~ \left|\frac{y}{(t+\sigma)^2}\right|\leq 1\right\},\\
    &\Omega_{2,3} := \left\{\left|\frac{y}{(t+\sigma)^2}\right|> 1\right\}.
\end{align*}
For $II_{k,1}$, by \eqref{Taylor}, we get
\begin{align*}
   II_{k,1}&\leq \frac{C}{(t+\sigma)^2}\int_{1\leq\frac{x-x_2}{(t+\sigma)^2}<3}|\epsilon(x)|dx\int_{\left|\frac{y}{(t+\sigma)^2}\right|\leq 1}(|\nabla R_k(x+\theta y)|+|R_k(x-y)|)dy\\
\leq& C\int_{1\leq\frac{x-x_2}{(t+\sigma)^2}<3}|\epsilon(x)|\frac{1}{|x-x_k|^2}dx\\
\leq&C\|\epsilon\|_{L^2}(t+\sigma)^{-3}.
\end{align*}
By the definition of $\Omega_{2,2}$ and $\Phi$, we can obtain $II_{k,2}=0$.

The term $II_{k,3}$ can be estimated by
\begin{align*}
  II_{k,3}\leq C\int_{\left|\frac{y}{(t+\sigma)^2}\right|> 1}|y|^{-2}dy\int |\epsilon(x)|(|R_k(x+y)|+|R_k(x-y)|) dx
\leq C(t+\sigma)^{-2}\|\epsilon\|_{L^2},
\end{align*}
where we have used that the H\"older inequality and  $\|Q\|_{L^2}$ is bounded.

Hence, we conclude that
\begin{align}\label{est-ii}
|II|\leq C(t+\sigma)^{-2}\|\epsilon\|_{L^2}.
\end{align}

{\bf $(iii)$ Estimate of $III$.}
At last, we consider the third term on the right-hand side of \eqref{du2-bc}.
By using the integration by parts formula,
\begin{align*}
   III=2{\rm Im}\int \bar{\epsilon} D\epsilon\Phi_2 dx={\rm Im}\int  \epsilon(D(\bar{\epsilon}\Phi_2)-D\bar{\epsilon}\Phi_2) dx.
\end{align*}
In view of the definition of $\Phi_k$,
we apply the Calder\'on estimate in Lemma \ref{lemma:commu:1} to get
\begin{align*}
\|D(\bar{\epsilon}\Phi_2)-D\bar{\epsilon}\Phi_2\|_{L^2}\leq C(t+\sigma)^{-2}\|\epsilon\|_{L^{2}},
\end{align*}
which yields
\begin{align}
    \label{est-iii}
|III|\leq C(t+\sigma)^{-2}\|\epsilon\|_{L^2}^2.
\end{align}

{\bf $(iv)$ Estimate of $IV$.} By the definition of $\Phi$, the parameter estimate \eqref{decom:small} and the conservation of mass, we can obtain
\begin{align}\label{est:iv}
    |IV|\leq C(t+\sigma)^{-2}A_0\left(\alpha+\frac{1}{\sigma}\right)\leq C(t+\sigma)^{-2}.
\end{align}

Now, combining the estimates \eqref{est-i}, \eqref{est-ii}, \eqref{est-iii} and \eqref{est:iv} together, we get
\begin{align*}
  \left|\frac{d}{dt}\int |u|^2\Phi_2dx\right|\leq C(t+\sigma)^{-2}\left(1+\|\epsilon\|_{L^2}+\|\epsilon\|_{L^2}^2\right).
\end{align*}
By integration between $0$ and $t$, we get
\begin{align*}
  \mathcal{J}_k(t)-\mathcal{J}_k(0)\leq \frac{C}{t+\sigma}\sup_{0<\tau<t}\|\epsilon(\tau)\|_{L^2}^2+\frac{C}{t+\sigma}.
\end{align*}
This completes the proof of Lemma \ref{lemma:mass:loc}.
%\hfill $\square$
\end{proof}

The analogue of Lemma \ref{lemma:J1} for the case of multi-solitary wave solutions is the following result.
\begin{lemma}
Let $G$ be defined as \eqref{def:G}. For all $t\in[0,t^*]$, we have
\begin{align}\label{expan:2}
     G(u(t))=&\sum_{k=1}^2G(Q_{\omega_k(0)})+H_K(\epsilon,\epsilon)+\|\epsilon(t)\|_{H^\frac{1}{2}}^2\beta\left(\|\epsilon(t)\|_{H^\frac{1}{2}}\right)\notag\\
     &+\sum_{k=1}^2\mathcal{O}(|\omega_k(t)-\omega_k(0)|^2)+\mathcal{O}\left(\langle \sigma\rangle^{-2}\right),
\end{align}
with $\beta(\epsilon)\to0$ as $\epsilon\to0$, where
\begin{align*}
H_K(\epsilon,\epsilon)= \frac{1}{2}\int|D^{\frac{1}{2}}\epsilon|^2+\sum_{k=1}^2\frac{\omega_k(0)}{2}\int|\epsilon|^2\Phi_k(t)-\sum_{k=1}^2\int \left(\frac{1}{2}|R_k|^{p-1}|\epsilon|^2+\frac{p-1}{2}|R_k|^{p-3}(\Re(R_k\epsilon))^2\right).
\end{align*}
\end{lemma}
\begin{proof}
The proof of this Lemma is similar to that of Lemma \ref{lemma:J1}. Here we omit it.
\end{proof}

The next lemma is the coercivity of $H_K$. In the multi-bubble case, it is important to derive the following localized version of the coercivity estimate in the construction of multi-solitary wave solutions and its stability.
\begin{lemma}\label{lemma:coer:m2}
There exists $\lambda_k>0$ such that
\begin{align*}
H_K(\epsilon(t),\epsilon(t))\geq\lambda_k\|\epsilon(t)\|_{H^{\frac{1}{2}}}^2.
\end{align*}
\end{lemma}
\begin{proof}
As in the proof of Lemma \ref{lemma:coer:2}, the proof of Lemma \ref{lemma:coer:m2} is based on Lemma \ref{lemma:linear:operator}. It also requires localization arguments. First, we give a localized version of Lemma \ref{lemma:coer:2}. Let $0<a<1$ and $\phi:\mathbb{R}\to\mathbb{R}$ be a $C^2$-even function such that $\phi^{\prime}\leq0$ with
\begin{align*}
    \phi(x)=1~~\text{on}~~[0,1],~~\phi(x)=|x|^{-a} ~~\text{on}~~[2,+\infty),~~\text{and}~~0<\phi<1.
\end{align*}
Let $R>0$ and $\phi_R(x)=\phi(x/R)$. Set
\begin{align*}
    H_{\phi_R}(v,v)=&\frac{1}{2}\int \phi_{R}(\cdot-x_0)\left\{\int|D^{\frac{1}{2}}v|^2+\omega_0\int|v|^2\right\}\\
    &-\int \left(\frac{1}{2}|Q_{\omega_0}(\cdot-x_0)|^{p-1}|v|^2+\frac{p-1}{2}|Q_{\omega_0}(\cdot-x_0)|^{p-3}(\Re(Q_{\omega_0}(\cdot-x_0)e^{-i\gamma_0}v))^2\right).
\end{align*}
{\bf Claim:} Let $\omega_0>0$, $x_0, \gamma_0\in\mathbb{R}$. Assume that there exists a solution $Q_{\omega_0}$ of \eqref{equ-elliptic}. There exists $R_0>2$ such that for all $R>R_0$, if $v\in H^{\frac{1}{2}}(\mathbb{R})$ satisfies
\begin{align*}
    \Re\int Q_{\omega_0}(\cdot-x_0)e^{-i\gamma_0}v=\Re\int Q_{\omega_0}^{\prime}(\cdot-x_0)e^{-i\gamma_0}v=\Im\int Q_{\omega_0}(\cdot-x_0)e^{-i\gamma_0}v=0.
\end{align*}
Then
\begin{align*}
    H_{\phi_R}(v,v)\geq C\int \phi_R(\cdot-x_0)\left(\left|D^{\frac{1}{2}}v\right|^2+|v|^2\right).
\end{align*}
Now we prove this {\bf Claim}. For the sake of simplicity, we assume that $x_0=0$ and $\gamma_0=0$. Set
\[ \tilde{v}:=v\phi_R^{\frac12}~~\text{and}~~\tilde{v}=\tilde{v}_1+i\tilde{v}_2.
\]
Then, we have
\begin{align} \label{g}
D^\frac{1}{2} v\phi_R^\frac{1}{2}=D^\frac{1}{2}\tilde{v}+\left(D^\frac{1}{2}(\tilde{v}\phi_R^{-{\frac{1}{2}}})-D^{\frac{1}{2}} \tilde{v}\phi_R^{-{\frac{1}{2}}}\right)\phi_R^{\frac{1}{2}}=:D^{\frac{1}{2}}\tilde{v}+h,
\end{align}
where
\begin{align}\label{def:h}
    h:=\left(D^\frac{1}{2}(\tilde{v}\phi_R^{-{\frac{1}{2}}})-D^{\frac{1}{2}} \tilde{v}\phi_R^{-{\frac{1}{2}}}\right)\phi_R^{\frac{1}{2}}.
\end{align}
It follows that
\begin{align}\label{m}
&\int (|D^{\frac{1}{2}} v|^2 +|v|^2)\phi_R -pQ_{\omega_0}^{p-1}v_1^2-Q_{\omega_0}^{2}v_2^2dx \notag\\
=&\int|D^{\frac{1}{2}} \tilde {v}|^2+|\tilde {v}|^2-pQ_{\omega_0}^{p-1} \tilde {v}_1^2-Q_{\omega_0}^{p-1} \tilde {v}_2^2dx\notag\\
&+ \int(1-\phi^{-1}_R)(pQ_{\omega_0}^{p-1}\tilde{v}_1^2+Q_{\omega_0}^{p-1}\tilde{v}_2^2)dx
 +\|h\|_{L^2}^2
+2{\Re} \langle D^{\frac{1}{2}}\tilde{v}, h\rangle  \nonumber \\
=&: K_1+K_2+K_3+K_4.
\end{align}
In the sequel,
let us estimate each term on the right-hand side of \eqref{m}.

{\bf $(i)$ Estimate of $K_1$.} We claim that there exists $C(R)>0$ with $\lim_{R\rightarrow+\infty}C(R)=0$ such that
\begin{align}\label{est-scal}
\left|\int Q_{\omega_0}\tilde{v}\right|+\left|\int Q_{\omega_0}^{\prime}\tilde{v}\right|\leq C(R)\|\tilde{v}\|_{L^2}^2. % This is equivalent to orthogonality condition
\end{align}
Thus, along with Lemma \ref{lemma:coer:2} and \eqref{est-scal} yields that there exists $C>0$ such that for $R$ sufficiently large,
\begin{align}\label{k1}
K_1=\int|D^{\frac{1}{2}} \tilde {v}|^2+|\tilde {v}|^2-pQ_{\omega_0}^{p-1} \tilde {v}_1^2-Q_{\omega_0}^{p-1} \tilde {v}_2^2dx\geq C\|\tilde{v}\|^2_{H^{\frac{1}{2}}}.
\end{align}
In order to prove \eqref{est-scal}, we rewrite
\begin{align*}
\int \tilde{v}_1 Q_{\omega_0}=\int v_1Q_{\omega_0}+\int \tilde{v}_1Q_{\omega_0}(\phi_R^{\frac{1}{2}}-1)\phi_R^{-{\frac{1}{2}}}dx.
\end{align*}
Notice that
\begin{align*}
\left|\int \tilde{v}_1Q_{\omega_0}\left(\phi_R^{\frac{1}{2}}-1\right)\phi_R^{-{\frac{1}{2}}}dx\right|\leq C\int_{|x|\geq R} \left|\tilde{v}_1Q_{\omega_0}\phi_R^{-{\frac{1}{2}}}\right|dx
\leq C\|\tilde{v}\|_{L^2}\left(\int_{|x|\geq R}Q_{\omega_0}^2\phi_R^{-1}dx\right)^{\frac{1}{2}}.
\end{align*}
By the decay property that $Q_{\omega_0}(y)\sim\langle y\rangle^{-2}$ (see \eqref{decay}),
\begin{align*}
\int_{|x|\geq R}Q_{\omega_0}^2\phi_R^{-1}dx=\frac{1}{R}\int_{|y|\geq 1}Q_{\omega_0}^2(Ry)\phi^{-1}(y)dy
\leq C\frac{1}{R^5}\int_{|y|\geq 1}|y|^{a-4}dy\leq \frac{C}{R^5}.
\end{align*}
Thus, we get
\begin{align*}
\left|\int \tilde{v}_1Q_{\omega_0}-\int v_1 Q_{\omega_0}\right|\leq CR^{-\frac{5}{2}}\|\tilde{v}\|_{L^2}.
\end{align*}
Similar arguments also apply to the term $\int Q_{\omega_0}^{\prime}\Tilde{v}$,
so we obtain \eqref{est-scal}, as claimed.

{\bf $(ii)$ Estimate of $K_2$.}
Using the decay property of $Q_{\omega_0}$ (see \eqref{decay}) again, we see that
\begin{align}
    \label{k2}
|K_2|=&\left|\int(1-\phi^{-1}_R)(pQ_{\omega_0}^{p-1}\tilde{v}_1^2+Q_{\omega_0}^{p-1}\tilde{v}_2^2)dx\right|\notag\\
\leq& C\int_{|x|\geq R}\phi^{-1}_RQ_{\omega_0}^{p-1}|\tilde{v}|^2dx\leq C\|\phi^{-1}_RQ_{\omega_0}^{p-1}\|_{L^\infty(|x|\geq R)}\|\tilde{v}\|_{L^2}^2\leq R^{-2(p-1)}\|\tilde{v}\|_{L^2}^2.
\end{align}

{\bf $(iii)$ Estimate of $K_3$ and $K_4$.} We claim that there exists $C(R)>0$ with $\lim_{A\rightarrow+\infty}C(R)=0$ such that
\begin{align}\label{est-g}
\|h\|_{L^2}\leq C(R)\|\tilde{v}\|_{L^2},
\end{align}
where $h$ is given by \eqref{def:h}.

This yields that
\begin{align}\label{k34}
|K_3|+|K_4|\leq C(R)\|\tilde{v}\|^2_{H^{\frac 12}}.
\end{align}

In order to prove \eqref{est-g},
by Lemma \ref{lem-formula}(see \eqref{id-fra2}) and the Minkowski inequality,
\begin{align*}
\|h\|_{L^2}&=C\left(\int \phi_R(x)\left|\int \frac{\tilde{v}(x+y)\left(\phi_R^{-\frac12}(x+y)-\phi_R^{-\frac12}(x)\right)
-\tilde{v}(x-y)\left(\phi_R^{-\frac12}(x)-\phi_R^{-\frac12}(x-y)\right)}
{|y|^{\frac{3}{2}}}dy\right|^2dx\right)
^{{\frac{1}{2}}}\\
& \leq C\int|y|^{-\frac32}\left(\int \phi_R(x)\left(\tilde{v}(x+y)\left(\phi_R^{-\frac12}(x+y)-\phi_R^{-\frac12}(x)\right)\right)^2dx\right)
^{\frac12}dy\\
& \leq C\sum_{j=1}^{4}\int_{\Omega_{j,1}}|y|^{-\frac32}
\left(\int_{\Omega_{j,2}} \phi_R(x)\left(\tilde{v}(x+y)\left(\phi_R^{-\frac12}(x+y)-\phi_R^{-\frac12}(x)\right)\right)^2dx\right)
^{\frac12}dy\\
&=:H_1+H_2+H_3+H_4,
\end{align*}
where $\Omega_1: =\Omega_{1,1} \cup \Omega_{1,2}=\left\{|x|\leq \frac{R}{2}, |y|\leq \frac{R}{4}\right\}$.
For the sake of simplicity,
we write $\Omega_j := \Omega_{j,1} \cup \Omega_{j,2}$ for the remaining three regimes, $2\leq j\leq 4$,
and take
$\Omega_2:=\left\{|x|\geq \frac{R}{2}, |y|\leq \frac{R}{4}\right\}$, $\Omega_3:=\left\{|x|\leq \frac{|y|}{2}, |y|\geq \frac{R}{4}\right\}$ and $\Omega_4:=\left\{|x|\geq \frac{|y|}{2}, |y|\geq \frac{R}{4}\right\}$.
Then, the proof of \eqref{est-g} is reduced to estimating $H_k$, $1\leq k\leq 4$.

First notice that $\phi_R^{-\frac12}(x+y)-\phi_R^{-\frac12}(x)=0$ for $(x, y)\in \Omega_1$,
and so  $H_1=0$.

For $H_2$. Using the mean valued theorem, we get that for some $0\leq \theta\leq 1$,
\begin{align*}
\phi_R(x)\left|\phi_R^{-\frac12}(x+y)-\phi_R^{-\frac12}(x)\right|^2=\frac{1}{4}\frac{\phi_R(x)}{\phi^3_R(x+\theta y)}|\phi^\prime_R(x+\theta y)|^2|y|^2.
\end{align*}
Since  $\frac{R}{2}\leq |x|\leq 3R$ and $|y|\leq \frac{R}{4}$, we have $\frac{R}{4}\leq |x+\theta  y|\leq 4R$,
and thus
\begin{align*}
\phi_R(x)\left|\phi_R^{-\frac12}(x+y)-\phi_R^{-\frac12}(x)\right|^2=\frac{1}{4R^2}\frac{\phi_R(x)}{\phi^3_R(x+\theta y)}\left|\phi^\prime\left(\frac{x+\theta y}{R}\right)\right|^2|y|^2\leq \frac{C}{R^2}|y|^2.
\end{align*}
While for $|x|\geq 3R$ and $|y|\leq \frac{R}{4}$,
we have $\frac{|x|}{2}\leq |x+\theta y|\leq \frac{3|x|}{2}$,
and thus
\begin{align*}
\phi_R(x)\left|\phi_R^{-\frac12}(x+y)-\phi_R^{-\frac12}(x)\right|^2
\leq C\frac{\phi_R(\frac x2)}{\phi^3_A(\frac {3x}{2})}|\phi^\prime_R(x+\theta y)|^2|y|^2
\leq \frac{C}{R^2}|y|^2.
\end{align*}
Hence, we obtain
\begin{align}
    \label{g2}
H_2\leq \frac{C}{R}\int_{0}^{\frac{R}{4}}y^{-{\frac{1}{2}}}dy\|\tilde{v}\|_{L^2}\leq CR^{-{\frac{1}{2}}}\|\tilde{v}\|_{L^2}.
\end{align}
For $H_3$, we shall use the fact that
\begin{align}
    \label{1}
\phi_R(x)\left|\phi_R^{-\frac12}(x+y)-\phi_R^{-\frac12}(x)\right|^2\leq C\phi_R(x)\left(\frac{1}{\phi_R(x+y)}+\frac{1}{\phi_R(x)}\right).
\end{align}

For $(x, y)\in \Omega_3$,
we have $\frac{|y|}{2}\leq |x+y|\leq \frac{3|y|}{2}$, so
\begin{align*}
\phi_R(x)\phi_R^{-1}(x+y)\leq C\phi_R^{-1}\left(\frac{3|y|}{2}\right)\leq C|y|^a,
\end{align*}
which implies
\begin{align*}
\phi_R(x)\left|\phi_R^{-\frac12}(x+y)-\phi_R^{-\frac12}(x)\right|^2\leq C(1+|y|^a),
\end{align*}
and thus
\begin{align}
    \label{g3}
H_3\leq C\int_{\frac{R}{4}}^{+\infty}y^{-\frac{3}{2}}\left(1+y^{\frac{a}{2}}\right)dy\|\tilde{v}\|_{L^2}
\leq CR^{\frac{a-1}{2}}\|\tilde{v}\|_{L^2}.
\end{align}
The estimate of $H_4$ also relies on \eqref{1}. In fact, for $(x, y)\in \Omega_4$, we have $|x+y|\leq 3|x|$,
and thus
\begin{align*}
\phi_R(x)\phi_R^{-1}(x+y)\leq C\phi_R(x)\phi_R^{-1}(3x)\leq C,
\end{align*}
and, by \eqref{1},
\begin{align*}
\phi_R(x)\left|\phi_R^{-\frac12}(x+y)-\phi_R^{-\frac12}(x)\right|^2\leq C.
\end{align*}
It follows that
\begin{align}\label{g4}
H_4\leq C\int_{\frac{R}{4}}^{+\infty}y^{-\frac{3}{2}}dy\|\tilde{v}\|_{L^2}\leq CR^{-{\frac{1}{2}}}\|\tilde{v}\|_{L^2}.
\end{align}
Thus, combining \eqref{g2}, \eqref{g3} and \eqref{g4} together, we obtain \eqref{est-g}, as claimed.

Now, plugging \eqref{k1}, \eqref{k2} and \eqref{k34} into \eqref{m}
we get that for $R$ large enough,
\begin{align*}
\int (|D^{\frac{1}{2}} v|^2 +|v|^2)\phi_R -pQ^{p-1}v_1^2-Q^{p-1}v_2^2dx
\geq C_1\|\tilde{v}\|_{H^{\frac{1}{2}}}^2.
\end{align*}
Hence, in order to finish the proof,
it remains to show that there exists $C>0$ such that for $R$ sufficiently large
\begin{align}\label{wtf-H12-esti}
\|\tilde{v}\|_{H^{\frac{1}{2}}}^2\geq C\int\left(|D^{\frac{1}{2}} v|^2+|v|^2\right)\phi_Rdx.
\end{align}
Since $\|\tilde{v}\|^2_{L^2}=\int|v|^2\phi_Rdx$ and
\begin{align*}
\|D^{\frac{1}{2}} \tilde{v}\|^2_{L^2}&=\int\left|D^{\frac{1}{2}} v\phi_R^{\frac{1}{2}}+\left(D^{\frac{1}{2}}\tilde{v}-D^{\frac{1}{2}}\left(\tilde{v}\phi_R^{-{\frac{1}{2}}}\right)\phi_R^{\frac{1}{2}}\right)\right|^2dx
=\int\left|D^{\frac{1}{2}} v\phi_R^{\frac{1}{2}}-h\right|^2dx\\
&=\int\left|D^{\frac{1}{2}} v\right|^2\phi_Rdx+\int|h|^2dx-2{\Re}\int D^{\frac{1}{2}} v\phi_R^{\frac{1}{2}}\bar{h}dx,
\end{align*}
where $h$ is given by \eqref{g}. Applying \eqref{est-g} to derive that for $R$ large enough
\begin{align*}
\|D^{\frac{1}{2}} \tilde{v}\|^2_{L^2}+\|\tilde{ v}\|^2_{L^2}\geq C\int\left(|D^{\frac{1}{2}} v|^2+|v|^2\right)\phi_Rdx.
\end{align*}
This yields \eqref{wtf-H12-esti} and finishes the proof of this Claim.

Now, we finish the proof of Lemma \ref{lemma:coer:m2}.  Let $R>R_0$ and $\sigma>0$. By the definition of $\Phi_k$, $k=1,2$, we decompose $H_K(\epsilon,\epsilon)$ as follows,
\begin{align*}
    H_K(\epsilon,\epsilon)=&\frac{1}{2}\sum_{k=1}^2\int\phi_R(\cdot-x_k)\left[|D^{\frac{1}{2}}\epsilon|^2+\omega_k(0)|\epsilon|^2\right]\\
    &-\frac{1}{2}\sum_{k=1}^2\int \left(|R_k|^{p-1}|v|^2+(p-1)|R_k|^{p-3}(\Re(\Bar{R}_kv))^2\right)\\
    &+\frac{1}{2}\sum_{k=1}^2\int(\Phi_k-\phi_R(\cdot-x_k))\omega_k(0)|\epsilon|^2+\frac{1}{2}\sum_{k=1}^2\int(1-\phi(\cdot-x_k))|D^{\frac{1}{2}}\epsilon|^2.
\end{align*}
From the claim, for any $k=1,2$, we have, for $R$ large enough,
\begin{align*}
    &\int\phi_R(\cdot-x_k)\left[|D^{\frac{1}{2}}\epsilon|^2+\omega_k(0)|\epsilon|^2\right]-\int \left(|R_k|^{p-1}|v|^2+(p-1)|R_k|^{p-3}(\Re(\Bar{R}_kv))^2\right)\\
    \geq& C_k\int \phi_R(\cdot-x_k(t))\left[|D^{\frac{1}{2}}\epsilon|^2+|\epsilon|^2\right].
\end{align*}
Moreover, by the properties of $\phi_R$ and $\Phi_k(t)$, for $\sigma$ large enough, where $\sigma$ is given by \eqref{def:sig0}, then we have
\begin{align*}
    \Phi_k-\phi_R(\cdot-x_k)\geq-|\sigma|^{-a}~~\text{and}~~1-\phi_R(\cdot-x_k)\geq-|\sigma|^{-a}.
\end{align*}
There exists  $\delta(k)=\delta(\omega_k)>0$ such that
\begin{align*}
    |D^{\frac{1}{2}}\epsilon|^2+\omega_k(0)|\epsilon|^2\geq\delta_k\left(|D^{\frac{1}{2}}\epsilon|^2+|\epsilon|^2\right)\geq0.
\end{align*}
Hence,
\begin{align*}
    &\int(\Phi_k-\phi_R(\cdot-x_k))\omega_k(0)|\epsilon|^2+\int(1-\phi(\cdot-x_k))|D^{\frac{1}{2}}\epsilon|^2\\
    \geq&\delta_k\int(\Phi_k-\phi_R(\cdot-x_k))\omega_k(0)|\epsilon|^2+\delta_k\int(1-\phi(\cdot-x_k))|D^{\frac{1}{2}}\epsilon|^2-|\sigma|^{-a}\int\left(|D^{\frac{1}{2}}\epsilon|^2+|\epsilon|^2\right).
\end{align*}
Combining the above estimates, we get
\begin{align*}
    H_K(\epsilon,\epsilon)\geq\lambda_k\int\sum_{k=1}^2\Phi_k\left(|D^{\frac{1}{2}}\epsilon|^2+|\epsilon|^2\right)-|\sigma|^{-a}\int\left(|D^{\frac{1}{2}}\epsilon|^2+|\epsilon|^2\right),
\end{align*}
where $\lambda_k=\min\{C_k,\delta_k\}$. By the definition of  $\Phi_k$, we obtain the desired result by taking $\sigma$ large enough. Now we complete the proof of Lemma \ref{lemma:coer:m2}.
\end{proof}

As in the proof of a single solitary waves stability result in section \ref{sec:s}, we now proceed in last two steps: first, we control the size of $\epsilon(t)$ in $H^\frac{1}{2}$, and second, we check that for any two, $|\omega_k(t)-\omega_k(0)|$ is quadratic in $|\epsilon(t)|$.

{\bf Step 3. Energetic control of $\|\epsilon(t)\|_{H^\frac{1}{2}}$.}  We give the following lemma:.
\begin{lemma}\label{lemma:energy}
Assume that $\sigma$ is given by \eqref{def:sig0}. For all $t\in[0,t^*]$, the following holds:
\begin{align}\label{est:small:s}
    \|\epsilon(t)\|_{H^\frac{1}{2}}^2 +|\mathcal{J}_2(t)-\mathcal{J}_2(0)|\leq&\frac{C}{\sigma}\sup_{t'\in[0,t]}\|\epsilon(t')\|_{L^2}^2+C\|\epsilon(0)\|_{H^\frac{1}{2}}^2\notag\\
    &+C\sum_{k=1}^2|\omega_k(t)-\omega_k(0)|+\frac{C}{\sigma}.
\end{align}
\end{lemma}
\begin{proof}
First, we write \eqref{expan:2} at $t>0$ and at $t=0$:
\begin{align*}
    E(u(t))+\mathcal{J}(t)=&E(Q_{\omega_k(0)})+\sum_{k=1}^2\mathcal{J}_{\omega_k(0)}(Q_{\omega_k(0)})+H_K(\epsilon(t),\epsilon(t))+\|\epsilon(t)\|_{H^\frac{1}{2}}^2\beta\left(\|\epsilon(t)\|_{H^\frac{1}{2}}\right)\notag\\
     &+\sum_{k=1}^2\mathcal{O}(|\omega_k(t)-\omega_k(0)|^2)+\mathcal{O}\left(\langle \sigma\rangle^{-2}\right),
\end{align*}
and
\begin{align*}
    &E(u(0))+\mathcal{J}(0)\\
    =&E(Q_{\omega_k(0)})+\sum_{k=1}^2\mathcal{J}_{\omega_k(0)}(Q_{\omega_k(0)})+H_K(\epsilon(0),\epsilon(0))+\|\epsilon(0)\|_{H^\frac{1}{2}}^2\beta\left(\|\epsilon(0)\|_{H^\frac{1}{2}}\right)+\mathcal{O}\left(\langle \sigma\rangle^{-2}\right).
\end{align*}
Since $E(u(t))=E(u(0))$ and $H_K(\epsilon(0),\epsilon(0))\leq C\|\epsilon(0)\|_{H^\frac{1}{2}}^2$, then from above equalities, we deduce that
\begin{align*}
    H_K(\epsilon(t),\epsilon(t))\leq&(\mathcal{J}(t)-\mathcal{J}(0))+C\|\epsilon(0)\|_{H^\frac{1}{2}}^2+C\sum_{k=1}^2(|\omega_k(t)-\omega_k(0)|^2)\\  &+C\|\epsilon(t)\|_{H^\frac{1}{2}}^2\beta\left(\|\epsilon(t)\|_{H^\frac{1}{2}}\right)+C\langle \sigma\rangle^{-2}.
\end{align*}
From the conservation of mass and Lemma \ref{lemma:coer:m2}, we obtain
\begin{align}\label{est:linear:1}
    \lambda_K\|\epsilon(t)\|_{H^\frac{1}{2}}^2\leq \left(\mathcal{J}_2(t)-\mathcal{J}_2(0)\right)+C\|\epsilon(0)\|_{H^{\frac{1}{2}}}^2+C\sum_{k=1}^2(|\omega_k(t)-\omega_k(0)|^2)+C\langle \sigma\rangle^{-2},
\end{align}
where we assume that $\alpha$ is small enough ($\alpha$ is as in \eqref{ass:sta:1}).

By Lemma \ref{lemma:mass:loc}, we have
\begin{align}\label{est:local:3}
    \mathcal{J}_2(t)-\mathcal{J}_2(0)\leq \frac{C}{\sigma}\sup_{0<\tau<t}\|\epsilon(\tau)\|_{L^2}^2+\frac{C}{\sigma}.
\end{align}
Injecting \eqref{est:local:3} into \eqref{est:linear:1}, we obtain
\begin{align*}
     \lambda_K\|\epsilon(t)\|_{H^\frac{1}{2}}^2\leq \frac{C}{\sigma}\sup_{t'\in[0,t]}\|\epsilon(t')\|_{L^2}^2+C\|\epsilon(0)\|_{H^\frac{1}{2}}^2+C\sum_{k=1}^2(|\omega_k(t)-\omega_k(0)|^2)+\frac{C}{\sigma}.
\end{align*}
Using this and \eqref{est:linear:1} again, we obtain
\begin{align*}
    |\mathcal{J}_2(t)-\mathcal{J}_2(0)|\leq &\frac{C}{\sigma}\sup_{t'\in[0,t]}\|\epsilon(t')\|_{L^2}^2+C\|\epsilon(0)\|_{H^\frac{1}{2}}^2\\
    &+C\sum_{k=1}^2(|\omega_k(t)-\omega_k(0)|^2)+\frac{C}{\sigma}.
\end{align*}
Then, from above estimates, \eqref{est:small:s} holds.
Now we have completed the proof of Lemma \ref{lemma:energy}.
\end{proof}

{\bf Step 4. Quadratic control of $|\omega_k(t)-\omega_k(0)|$.} We give the following result.

Let $\Phi^+: \R\rightarrow [0,1]$ be a smooth function  such that
$|(\Phi^+)^\prime(x)| \leq C \sigma^{-1}$ for some $C>0$,
$\Phi^+(x)=0$ for $x\leq \frac{3}{2}$
and $\Phi^+(x)=1$ for $x\geq \frac{5}{2}$.
The localization functions $\Phi^+_k$, $k=1,2$ are defined by
\begin{align*}
    \Phi^+_1(x)\equiv 1~~\text{and}~~ \Phi^+_2(x) :=\Phi^+\left(\frac{x-x_2}{(t+\sigma)^2}\right).
\end{align*}

 In addition, we define $\Phi^-: \R\rightarrow [0,1]$ as a smooth function such that
$|(\Phi^-)^\prime(x)| \leq C \sigma^{-1}$ for some $C>0$,
$\Phi^-(x)=0$ for $x\leq \frac{1}{2}$
and $\Phi^-(x)=1$ for $x\geq \frac{3}{2}$.
The localization functions $\Phi^-_k$, $k=1,2$ are defined by
\begin{align*}
   \Phi^-_1(x)\equiv 1~~\text{and}~~ \Phi^-_2(x) :=\Phi^-\left(\frac{x-x_2}{(t+\sigma)^2}\right).
\end{align*}

Define
\begin{align}\label{def:Jk+}
    \mathcal{J}_k^{\pm}(t)=\frac{1}{2}\omega_k(0)\int|u(t,x)|^2\Phi^{\pm}_k(x)dx.
\end{align}
Hence, from Lemma \ref{lemma:mass:loc}, we have the following corollary:
\begin{corollary}\label{coro:12}
    Let $\mathcal{J}_2^{\pm}$ be defined as \eqref{def:Jk+}, then we have
   \begin{align*}
\mathcal{J}_2^{\pm}(t)-\mathcal{J}_2^{\pm}(0)\leq \frac{C}{\sigma}\sup_{0<\tau<t}\|\epsilon(\tau)\|_{L^2}^2+\frac{C}{\sigma}.
\end{align*}
\end{corollary}
Now we have the following estimates:
\begin{lemma}
Let $\mathcal{J}_k$ and $\mathcal{J}_k^{\pm}$ be defined as \eqref{def:JK} and \eqref{def:Jk+}, respectively. Then we have
    \begin{align}\label{claim:1}
    \left|\mathcal{J}^{+}_k(t)-\mathcal{J}_k(t)-\frac{1}{2}\sum_{k'=k}^2\int Q_{\omega_{k'}(t)}^2\right|\leq C\|\epsilon(t)\|_{H^\frac{1}{2}}^2+\langle \sigma\rangle^{-3
    },\\\label{claim:2}
    \left|\mathcal{J}_k(t)-\mathcal{J}^-_k(t)-\frac{1}{2}\sum_{k'=k}^2\int Q_{\omega_{k'}(t)}^2\right|\leq C\|\epsilon(t)\|_{H^\frac{1}{2}}^2+\langle \sigma\rangle^{-3}.
\end{align}
\end{lemma}
\begin{proof}
By the calculation, we have
\begin{align*}
\mathcal{J}^{+}_k(t)-\mathcal{J}_k(t)=&\omega_1(0)\int|u(t,x)|^2\Phi^{+}_k(x)dx-\frac{1}{2}\omega_1(0)\int|u(t,x)|^2\Phi_k(x)dx\\
=&\omega_k(0)\left(\frac{1}{2}\int|u(t,x)|^2\Phi^{+}_k(x)+\frac{1}{2}\int|u(t,x)|^2\left(\Phi^{+}_k(x)-\Phi_k\right)\right).
\end{align*}
By the definition of $\Phi$ and $\Phi^{+}$, and the decay property of $Q$, we have
\begin{align*}
    \int|R(t,x)|^2\left(\Phi^{+}_k(x)-\Phi_k\right)\leq C (t+\sigma)^{-3}.
\end{align*}
Therefore,
\begin{align*}
   \int|u(t,x)|^2\left(\Phi^{+}_k(x)-\Phi_k\right)\leq C\|\epsilon(t)\|_{H^\frac{1}{2}}^2+ C (t+\sigma)^{-3},
\end{align*}
where we used the orthogonality condition $\Re\int R_{k}(t)\epsilon(t)=0$.

On the other hand, by the orthogonality condition $\Re\int R_{k'}(t)\epsilon(t)=0$ and the algebraic decay of the $Q_{\omega_{k'}}$ (see \eqref{decay}), we have
\begin{align*}
   \left|\int|u(t,x)|^2\Phi^{+}_k-\sum_{k'=k}^2\int Q_{\omega_{k'}(t)}^2\right|\leq C\|\epsilon(t)\|_{H^\frac{1}{2}}^2+C \langle \sigma\rangle^{-3}.
\end{align*}
Hence, by the above estimates, we get \eqref{claim:1} holds. Using an argument similar to \eqref{claim:1}, we can easily obtain \eqref{claim:2}.
\end{proof}

\begin{lemma}\label{lemma:para:larg}
Assume that $\sigma$ is given by \eqref{def:sig0}. For all $t\in[0,t^*]$, it holds,
\begin{align}\label{est:para:larg}
    \sum_{k=1}^2|\omega_k(t)-\omega_k(0)|\leq C\sup_{t'\in[0,t]}\|\epsilon(t)\|_{H^\frac{1}{2}}^2+\frac{C}{\sigma}.
\end{align}
\end{lemma}
\begin{proof}
From Lemma \ref{lemma:mass:loc} and Corollary \ref{coro:12}, we have
\begin{align}\label{est:p:1}
    |\mathcal{J}_2(t)-\mathcal{J}_2(0)|\leq C\sup_{t'\in[0,t]}\|\epsilon(t)\|_{H^\frac{1}{2}}^2+C\sum_{k=1}^2|\omega_k(t)-\omega_k(0)|^2+C\langle \sigma\rangle^{-2},
\end{align}
and
\begin{align}\label{est:p:2}
  \left|\mathcal{J}^{\pm}_2(t)-\mathcal{J}^{\pm}_2(0)\right|\leq& \frac{C}{\sigma}\sup_{t'\in[0,t]}\|\epsilon(t')\|_{L^2}^2+\frac{C}{\sigma}.
\end{align}
Combining \eqref{claim:1}, \eqref{est:p:1} and \eqref{est:p:2}, we have,
\begin{align*}
    \left(\int Q_{\omega_2(t)}^2-Q_{\omega_2(0)}^2\right)\leq& 2\left[(\mathcal{J}_2^+(t)-\mathcal{J}_2(t))-(\mathcal{J}_2^+(0)-\mathcal{J}_2(0))\right]+C\sup_{t'\in[0,t]}\|\epsilon(t)\|_{H^\frac{1}{2}}^2+C\langle \sigma\rangle^{-2}\\
    \leq&2\left[(\mathcal{J}_2^+(t)-\mathcal{J}_2(t))]+2[(\mathcal{J}_2^+(0)-\mathcal{J}_2(0))\right]+C\sup_{t'\in[0,t]}\|\epsilon(t)\|_{H^\frac{1}{2}}^2+C\langle \sigma\rangle^{-2}\\
    \leq&C\sum_{k=1}^2|\omega_k(t)-\omega_k(0)|^2+C\sup_{t'\in[0,t]}\|\epsilon(t)\|_{H^\frac{1}{2}}^2+\frac{C}{\sigma^2}.
\end{align*}
Similarly,
\begin{align*}
    -\left(\int Q_{\omega_2(t)}^2-Q_{\omega_2(0)}^2\right)\leq& 2\left[(\mathcal{J}_2^-(t)-\mathcal{J}_2(t))-(\mathcal{J}_2^-(0)-\mathcal{J}_2(0))\right]+C\sup_{t'\in[0,t]}\|\epsilon(t)\|_{H^\frac{1}{2}}^2+C\langle \sigma\rangle^{-2}\\
    \leq&2\left[(\mathcal{J}_2^-(t)-\mathcal{J}_2(t))]+2[(\mathcal{J}_2^-(0)-\mathcal{J}_2(0))\right]+C\sup_{t'\in[0,t]}\|\epsilon(t)\|_{H^\frac{1}{2}}^2+C\langle \sigma\rangle^{-2}\\
    \leq&C\sum_{k=1}^2|\omega_k(t)-\omega_k(0)|^2+C\sup_{t'\in[0,t]}\|\epsilon(t)\|_{H^\frac{1}{2}}^2+\frac{C}{\sigma^2}.
\end{align*}
Therefore, we deduce, for all $k=2,\ldots,K$,
\begin{align}\label{est:p:3}
    \left|\left(\int Q_{\omega_2(t)}^2-Q_{\omega_2(0)}^2\right)\right|\leq C\sum_{k=1}^2|\omega_k(t)-\omega_k(0)|^2+C\sup_{t'\in[0,t]}\|\epsilon(t)\|_{H^\frac{1}{2}}^2+\frac{C}{\sigma}.
\end{align}
On the other hand, by the mass conservation \eqref{Mass} and the orthogonality conditions on $\epsilon$ (see \eqref{ortho:condition}), we have
\begin{align*}
    \left|\left(\int Q_{\omega_1(t)}^2-Q_{\omega_1(0)}^2\right)\right|\leq C\sup_{t'\in[0,t]}\|\epsilon(t)\|_{H^\frac{1}{2}}^2+C\langle \sigma\rangle^{-1}.
\end{align*}
This means that \eqref{est:p:3} is true for $k=1$.

Recall that $\omega_k(t)$, $\omega_k(0)$ are close to $\omega_k^0$ (see \eqref{decom:small}), then for any $k=1,2$,
\begin{align}\label{est:p:5}
    |\omega_k(t)-\omega_k(0)|\leq C\left|\int Q_{\omega_k(t)}^2-Q_{\omega_k(0)}^2\right|.
\end{align}
In particular, from \eqref{est:p:3} and \eqref{est:p:5}, we obtain
\begin{align*}
    |\omega_2(t)-\omega_2(0)|\leq C\sum_{k=1}^2|\omega_k(t)-\omega_k(0)|^2+C\sup_{t'\in[0,t]}\|\epsilon(t)\|_{H^\frac{1}{2}}^2+\frac{C}{\sigma}.
\end{align*}
Then, by a backward induction argument on $k$, using \eqref{est:p:3} and \eqref{est:p:5}, we deduce,
\begin{align*}
    |\omega_1(t)-\omega_1(0)|\leq C\sum_{k=1}^2|\omega_k(t)-\omega_k(0)|^2+C\sup_{t'\in[0,t]}\|\epsilon(t)\|_{H^\frac{1}{2}}^2+\frac{C}{\sigma}.
\end{align*}
Thus, for any $k=1,2$,
\begin{align*}
    |\omega_k(t)-\omega_k(0)|\leq C\sup_{t'\in[0,t]}\|\epsilon(t)\|_{H^\frac{1}{2}}^2+\frac{C}{\sigma}.
\end{align*}
This completes the proof of Lemma \ref{lemma:para:larg}.
\end{proof}

\begin{proof}[\textbf{Conclusion of the proof of Proposition \ref{proposition:1}.}] Combining the conclusions of the Lemma \ref{lemma:energy} and Lemma \ref{lemma:para:larg}, we obtain, for all $t\in[0,t^*]$,
\begin{align*}
    \|\epsilon\|_{H^\frac{1}{2}}^2\leq \frac{C}{\sigma}\sup_{t'\in[0,t]}\|\epsilon(t')\|_{L^2}^2+C\sup_{t'\in[0,t]}\left[\beta\left(\|\epsilon(t)\|_{H^\frac{1}{2}}\right)\|\epsilon(t)\|_{H^\frac{1}{2}}^2\right]+C\|\epsilon(0)\|_{H^\frac{1}{2}}^2+\frac{C}{\sigma}.
\end{align*}
For $\sigma$ large enough, we have for all $t\in[0,t^*]$,
\begin{align*}
    \|\epsilon(t)\|_{H^\frac{1}{2}}^2\leq \frac{1}{2}\sup_{t'\in[0,t]}\|\epsilon(t)\|_{H^\frac{1}{2}}^2+\|\epsilon(0)\|_{H^\frac{1}{2}}^2+\frac{C}{\sigma}.
\end{align*}
and so, for all $t\in[0,t^*]$,
\begin{align*}
     \|\epsilon(t)\|_{H^\frac{1}{2}}^2\leq C\|\epsilon(0)\|_{H^\frac{1}{2}}^2+\frac{C}{\sigma}.
\end{align*}
Using \eqref{est:para:larg}, we obtain
\begin{align*}
    \|\epsilon(t)\|_{H^\frac{1}{2}}^2+\sum_{k=1}^2|\omega_k(t)-\omega_k(0)|\leq C\|\epsilon(0)\|_{H^\frac{1}{2}}^2+\frac{C}{\sigma}.
\end{align*}
By \eqref{decom:small} and \eqref{decom:4}, we obtain
\begin{align*}
   \|\epsilon(t)\|_{H^\frac{1}{2}}^2+\sum_{k=1}^2|\omega_k(t)-\omega_k(0)|+\sum_{k=1}^2|\omega_k(0)-\omega_k^0|\leq C\alpha^2+\frac{C}{\sigma}.
\end{align*}
where $C$ is independent of $A_0$. To conclude the proof, we go back to $u(t)$,
\begin{align*}
    &\left\|u(t)-\sum_{k=1}^2Q_{\omega_k^0}(x-x_k(t))e^{i\gamma_k(t)}\right\|_{H^\frac{1}{2}}\\
    \leq&\left\|u(t)-\sum_{k=1}^2R_k(t)\right\|_{H^\frac{1}{2}}+\sum_{k=1}^2\left\|R_k(t)-Q_{\omega_k^0}(x-x_k(t))e^{i\gamma_k(t)}\right\|_{H^\frac{1}{2}}\\
    \leq&\|\epsilon(t)\|_{H^\frac{1}{2}}+C\sum_{k=1}^2|\omega_k(t)-\omega_k^0|\\
    \leq&\|\epsilon(t)\|_{H^\frac{1}{2}}+C\sum_{k=1}^2|\omega_k(t)-\omega_k(0))|+C\sum_{k=1}^2|\omega_k(0)-\omega_k^0|\\
    \leq&C_1\left(\alpha+\frac{C}{\sigma}\right).
\end{align*}
Notice that $C>0$ does not depend on $A_0$. Thus, we can choose $A_0=2C_1$, then $\alpha_0>0$ small enough and $\sigma_0$ large enough, and we obtain the conclusion of Proposition \ref{proposition:1}.
\end{proof}

\appendix
\section{Appendix}
In this section, we give the following local well-posedness result concerning the Cauchy problem
for half-wave equation \eqref{equ-hf}.
\begin{lemma}\label{Lemma:local}
Let $s>\frac{1}{2}$ be given. For every initial datum $u_0\in H^s(\mathbb{R})$, there
exists a unique solution $u\in C^0([t_0, T); H^s(\mathbb{R})$ of problem \eqref{equ-hf}. Here $t_0<T(u_0)\leq+\infty$ denotes its maximal time of existence (in forward time). Moreover,
we have the following properties.

(i) Conservation of $L^2$-mass, energy and linear momentum: It holds that
\begin{align*}
    M(u)=\int|u|^2,~~E(u)=\frac{1}{2}\int|D^{\frac{1}{2}}u|^2-\frac{1}{p+1}\int|u|^{p+1},~~P(u)=\int\bar{u}(-i\partial_x u),
\end{align*}
are conserved along the flow.

(ii) Blowup alternative in $H^{s}$: Either $T(u_0)=+\infty$ or, if $T(u_0)<+\infty$, then $\|u(t)\|_{H^{s}}\to\infty$ as $t\to T^-$.

(iii) Continuous dependence: The flow map $u_0\mapsto u(t)$ is Lipschitz continuous on bounded subsets of $H^s(\mathbb{R})$.

(iv) Global Existence: If $1<p<3$, then $T(u_0)=+\infty$ holds true.
\end{lemma}
\begin{proof}
    By the similar argument as \cite[Lemma D.1]{KLR2013ARMA}, we can obtain this lemma. Here we omit it.
\end{proof}
\renewcommand{\proofname}{\bf Proof.}

\noindent
{\bf Acknowledgments.}

%  V. Georgiev was partially supported by   Gruppo Nazionale per l'Analisi Matematica 2020, by the project PRIN  2020XB3EFL with the Italian Ministry of Universities and Research, by Institute of Mathematics and Informatics, Bulgarian Academy of Sciences, by Top Global University Project, Waseda University and the Project PRA 2022 85 of University of Pisa.
 This work was
 supported by  the National Natural Science
Foundation of China (12301090).
%and 11801519)
% and Fundamental Research Program of Shanxi Province NO.202203021222126.
%  Part of this work was done by Y.L. during his postdoctoral studies at Central China Normal University, and he would like to thank the professor Shuangjie Peng for fruitful discussions and constant encouragement.
%Q. Wang was partially supported by  the National Natural Science Foundation of China (11801519).
%%%%%%%%%%%%%%%%%%%%%%%%%%%%%%%%%%%%%%%%%%%%%%

\vspace{0.2cm}
\noindent
{\bf Data Availability}

We do not analyse or generate any datasets, because our work proceeds within a theoretical
and mathematical approach.

\vspace{0.2cm}
\noindent
{\bf Conflict of interest}

The authors declare no potential conflict of interest.

%\noindent{\color{red}\rule{17.5cm}{2pt}}

\vspace*{.5cm}

%\noindent{\color{red}\rule{17.5cm}{2pt}}

%\bibliographystyle{acm}
% \bibliographystyle{plain}
% \bibliography{ref}

\bigskip

% \begin{flushleft}
% Vladimir Georgiev,\\
% Dipartimento di Matematica, Universit\`{a} di Pisa, Largo B. Pontecorvo 5, 56127 Pisa, Italy\\
%  Faculty of Science and Engineering, Waseda University, 3-4-1, Okubo, Shinjuku-ku, Tokyo 169-8555, Japan\\
%  IMICBAS, Acad. Georgi Bonchev Str., Block 8, 1113 Sofia, Bulgaria\\
% E-mail: georgiev@dm.unipi.it
% \end{flushleft}

\begin{flushleft}
Yuan Li,\\
School of Mathematics and Statistics, Lanzhou University,
Lanzhou, 730000, Gansu, Peoples Republic of China,\\
E-mail: li\_yuan@lzu.edu.cn
\end{flushleft}

% \begin{flushleft}
% Kai Wang,\\
% College of Mathematics, Taiyuan University of Technology, Taiyuan,
% PR China\\
% E-mail: wangkai03@tyut.edu.cn
% \end{flushleft}

% \begin{flushleft}
% Qingxuan Wang,\\
% Department of Mathematics,
% Zhejiang Normal University,
% Jinhua, 321004, Zhejiang,
% PR China\\
% E-mail: wangqx@zjnu.edu.cn
% \end{flushleft}
\bigskip

\medskip

\end{document}